\documentclass[11pt,oneside,reqno,a4paper]{amsart}

\usepackage[margin=1.00in]{geometry}

\allowdisplaybreaks[3]

\usepackage{lmodern}

\usepackage{amsmath}
\usepackage{amsthm}
\usepackage{amssymb}
\usepackage{verbatim}
\usepackage{esint}
\usepackage{pdfpages}
\usepackage{tikz}
\usepackage{xcolor}

\usepackage[colorlinks]{hyperref}
\definecolor{myblue}{rgb}{0.0, 0.0, 1.0}
\definecolor{mygreen}{rgb}{0.2,0.8,0.2}
\hypersetup{
	colorlinks=true,
	linkcolor=myblue,         
	citecolor=mygreen,
	urlcolor=mygreen,
	linktocpage=true
}

\renewcommand{\Tilde}{\widetilde}
\renewcommand{\Hat}{\widehat}

\newcommand{\dd}{\mathrm{d}}
\newcommand{\cD}{\mathcal{D}}

\newcommand{\cV}{\mathcal{V}}
\newcommand{\N}{\mathbb{N}}
\newcommand{\R}{\mathbb{R}}
\renewcommand{\H}{\mathcal{H}}
\newcommand{\eps}{\varepsilon}
\newcommand{\specess}{\mathrm{spec}_{\mathrm{ess}}}
\newcommand{\spec}{\mathrm{spec}}

\newtheorem{theorem}{Theorem}
\newtheorem{lemma}[theorem]{Lemma}

\newtheorem{corollary}[theorem]{Corollary}
\theoremstyle{definition}
\newtheorem{definition}[theorem]{Definition}
\newtheorem{remark}[theorem]{Remark}

\title[]{Laplacian eigenvalues for large negative Robin parameters on domains with outward peaks}

\author[]{Konstantin Pankrashkin}
\address{Konstantin Pankrashkin: Carl von Ossietzky Universit\"{a}t Oldenburg,
	Fakult\"{a}t V, Institut f\"{u}r Mathematik,
	Ammerl\"ander Heerstra{\ss}e 114--118,
	26129 Oldenburg, Germany\\ 
	ORCID: \url{https://orcid.org/0000-0003-1700-7295}}
\email{konstantin.pankrashkin@uol.de}

\author[]{Firoj Sk}
\address{Firoj Sk: Carl von Ossietzky Universit\"{a}t Oldenburg,
	Fakult\"{a}t V, Institut f\"{u}r Mathematik,
	Ammerl\"ander Heerstra{\ss}e 114--118,
	26129 Oldenburg, Germany\\
	ORCID: \url{https://orcid.org/0000-0002-2653-5674}}
\email{firoj.sk@uol.de}	

\author[]{Marco Vogel}
\address{Marco Vogel: Technische Universit\"at Dortmund, Fakult\"at für Mathematik, Vogelpothsweg 87, 	44227 Dortmund, Germany}
\email{marco.vogel@math.tu-dortmund.de}



\begin{document}

\begin{abstract}
	We study the asymptotic behavior of individual eigenvalues of the Laplacian in domains with outward peaks for large negative Robin parameters. A large class of cross-sections is allowed, and the resulting asymptotic expansions reflect both the sharpness of the peak and the geometric shape of its cross-section. The results are an extension of previous works dealing with peaks whose cross-sections are balls.
\end{abstract}

	\maketitle

	\renewcommand{\proofname}{\bf Proof}

\section{Introduction}

For $\Omega\subset\R^d$ a bounded domain (non-empty connected open set) with suitably regular boundary,
we are interested in the following Robin eigenvalue problem:
\begin{equation}
	\label{robin}
\left\{
	\begin{aligned}
		-\Delta u&=\lambda u \text{ in }\Omega,\\
		\partial_\nu u&=\alpha u \text{ on }\partial\Omega,
	\end{aligned}	
\right.
\end{equation}
where $\alpha>0$ is a parameter and $\partial_\nu$ means the outer normal derivative. The problem \eqref{robin} is understood in a weak sense. For a strict formulation we denote by $\H^m$ the $m$-dimensional Hausdorff measure, and for the sake of readability we denote a function in $\Omega$ and its Sobolev trace on $\partial \Omega$ by the same symbol.
Under appropriate regularity assumptions on $\Omega$ the symmetric bilinear form $r^\Omega_\alpha$
defined on the domain $\cD(r^\Omega_\alpha)=W^{1,2}(\Omega)$ by
\[
r^\Omega_\alpha(u,v):=\int_\Omega \langle\nabla u,\nabla v\rangle_{\R^d}\,\dd\H^d -\alpha\int_{\partial \Omega}uv\,\dd \H^{d-1}
\]
is closed and generates a self-adjoint operator
$R^\Omega_\alpha$ with compact resolvent in $L^2(\Omega)$. Informally, the operator $R^\Omega_\alpha$ corresponds to the Laplacian $u\mapsto -\Delta u$ acting on the functions $u$ satisfying the boundary condition $\partial_\nu u=\alpha u$ on $\partial\Omega$, and the eigenvalues $\lambda$ in \eqref{robin} are understood as the eigenvalues of $R^\Omega_\alpha$. In particular, a number $\lambda$ is an eigenvalue of \eqref{robin} with an eigenfunction $u$ if and only if
\[
r^\Omega_\alpha(u,v)=\lambda \int_\Omega uv\,\dd\H^d \text{ for all } v\in W^{1,2}(\Omega).
\]
Alternatively, one can consider $u\mapsto r^\Omega_\alpha(u,u)$ as the energy functional for \eqref{robin} and characterize the eigenvalues variationally 
using the min-max principle. Remark that the boundary term in the above expression for $r^\Omega_\alpha(u,u)$ is negative for $\alpha>0$, so this case is usually termed as the case of \emph{negative} Robin parameters. 

In what follows, for a lower semibounded self-adjoint operator $A$ we denote by $\lambda_j(A)$ its $j$-th eigenvalue (if it exists), assuming that the eigenvalues are enumerated in the non-decreasing order  by taking into account their multiplicities. The goal of the present paper is to study the asymptotic behavior of the individual eigenvalues of $R^\Omega_\alpha$, i.e. of $\lambda_j(R^\Omega_\alpha)$ with fixed $j$, for $\alpha\to+\infty$ and a special class of domains $\Omega$. The dependence of $\lambda_j(R_\alpha^\Omega)$ on $\alpha$ for various classes of $\Omega$ has been given a considerable attention during the last decade, see the reviews in \cite{DFK,kobp}. While the above problem is linear, it originally appeared
in the study of non-linear problems, see e.g. \cite{DC} for a stochastic interpretation,
\cite{GS} for a discussion of an interplay between the Robin eigenvalues and critical temperature estimates in surface
superconductivity, and \cite{LOS} for a link with the study of reaction-diffusion equations.

 If $\Omega$ is a bounded Lipschitz domain, then all required regularity assumptions are satisfied, and it is known that for large $\alpha$
one has a two-sided bound
\begin{equation}
	\label{lipschitz}
-c\alpha^2\le \lambda_1(R_\alpha^\Omega)\le-\alpha^2
\end{equation}
with some $c\ge 1$, see e.g.~\cite[Prop.~4.12]{DFK} and~\cite[Lem.~2.7]{kobp},
and
the very recent paper \cite{dp} shows that no asymptotics of the form $\lambda_1(R_\alpha^\Omega)\sim -c_\Omega \alpha^2$ with $c_\Omega>0$ is available in general. Under stronger regularity assumptions
one can construct detailed asymptotic expansions for $\lambda_j(R^\Omega_\alpha)$ with any fixed $j$ involving various geometric properties of $\Omega$ and $\partial\Omega$, see e.g. \cite{antunes,bp,dk,HK,kr,LP,pp}
and the reviews in \cite{DFK,kobp}. On the other hand,
it was observed in \cite{LP} that \eqref{lipschitz} fails for domains $\Omega$ with outward peaks, which was later studied in greater detail in \cite{kov}. As the subsequent text is specifically devoted to the study of $\lambda_j(R^\Omega_\alpha)$ for such $\Omega$, let us introduce an adapted language in order to continue the discussion.

From now on let $d\in\N$ with $d\ge 2$. The vectors $x\in\R^d$ will be written in the form $x=(x_1,x')$ with  $x'\in\R^{d-1}$. The following definition is in the spirit of \cite[Sec.~5.1.1]{mazya}, see Figure~\ref{fig1} for an illustration:

\begin{definition}
	 \label{defin1}
Let $q>1$ and $\omega\subset\R^{d-1}$ be a bounded Lipschitz domain.	 
We say that a bounded domain $\Omega\subset\R^d$ has an outward peak at the origin, of sharpness order $q$
with cross-section $\omega$, if for some $\delta>0$ it holds
\begin{equation*}
	\label{eqpeak}
\Omega\cap (-\delta,\delta)^d=\Big\{x\in\R^{d}:\ x_1\in(0,\delta),\ x'\in x_1^q\omega\Big\}
\end{equation*}
and $\Omega$ is Lipschitz at all boundary points except at the origin.
\end{definition}

\begin{figure}
	\centering

	\tikzset{every picture/.style={line width=0.75pt}} 
	
	\scalebox{0.8}{\begin{tikzpicture}[x=0.75pt,y=0.75pt,yscale=-1,xscale=1]
			
			\draw  [fill={rgb, 255:red, 155; green, 155; blue, 155 }  ,fill opacity=1 ][line width=1.5]  (136.5,164) .. controls (222.5,164) and (251.5,75) .. (273.5,73) .. controls (385.5,89) and (392.5,82) .. (391.5,187) .. controls (371.5,238) and (278.25,280.75) .. (256.5,260) .. controls (234.75,239.25) and (225.5,170) .. (136.5,164) -- cycle ;
			\draw [color={rgb, 255:red, 128; green, 128; blue, 128 }  ,draw opacity=1 ]   (136.5,261) -- (135.51,48) ;
			\draw [shift={(135.5,46)}, rotate = 89.73] [color={rgb, 255:red, 128; green, 128; blue, 128 }  ,draw opacity=1 ][line width=0.75]    (10.93,-3.29) .. controls (6.95,-1.4) and (3.31,-0.3) .. (0,0) .. controls (3.31,0.3) and (6.95,1.4) .. (10.93,3.29)   ;
			\draw    (190,150.6) -- (189.6,37.8) -- (240.8,90.6) -- (240.3,280.4) -- (189.6,229.4) -- (190,180.2) ;
			\draw [color={rgb, 255:red, 128; green, 128; blue, 128 }  ,draw opacity=1 ]   (136.5,164) -- (432.5,164) ;
			\draw [shift={(434.5,164)}, rotate = 180] [color={rgb, 255:red, 128; green, 128; blue, 128 }  ,draw opacity=1 ][line width=0.75]    (10.93,-3.29) .. controls (6.95,-1.4) and (3.31,-0.3) .. (0,0) .. controls (3.31,0.3) and (6.95,1.4) .. (10.93,3.29)   ;
			\draw  [fill={rgb, 255:red, 128; green, 128; blue, 128 }  ,fill opacity=1 ] (215.6,132.6) .. controls (199.2,150.6) and (193.6,168.8) .. (218.4,202.8) .. controls (238,196) and (240.4,142.4) .. (215.6,132.6) -- cycle ;
			
			\draw (405,140.4) node [anchor=north west][inner sep=0.75pt]  [color={rgb, 255:red, 128; green, 128; blue, 128 }  ,opacity=1 ]  {$x_{1}$};
			\draw (142,55.4) node [anchor=north west][inner sep=0.75pt]  [color={rgb, 255:red, 128; green, 128; blue, 128 }  ,opacity=1 ]  {$x'$};
			\draw (296,97.4) node [anchor=north west][inner sep=0.75pt]    {$\Omega $};
			\draw (174.4,262.2) node [anchor=north west][inner sep=0.75pt]    {$x_{1} =c$};
			\draw (203.2,150.2) node [anchor=north west][inner sep=0.75pt]    {$c^{q} \omega $};

		\end{tikzpicture}
	}

	\caption{An example of a domain $\Omega$ with an outward peak with cross-section $\omega$.}\label{fig1}
\end{figure}
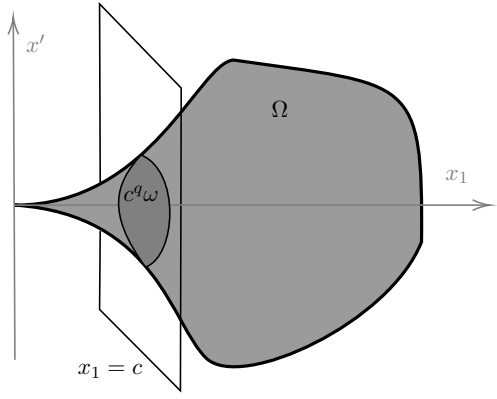

The main result of the present work is as follows:
\begin{theorem}\label{thmain}
Let $\Omega\subset\R^d$	be as in Definition~\ref{defin1} with some bounded Lipschitz cross-section $\omega\subset\R^{d-1}$ and sharpness order $q\in(1,2)$, then for any fixed $j\in\N$ one has
	\begin{equation}
		 \label{ljj}
	\lambda_j(R^\Omega_\alpha)=\left(\dfrac{\H^{d-2}(\partial\omega)}{\H^{d-1}(\omega)}\right)^\frac{2}{2-q} \lambda_j(L_1)\alpha^{\frac{2}{2-q}} +O\left(\alpha^{\frac{2}{2-q}-(q-1)}\right) \text{ for }\alpha\to +\infty,
	\end{equation}
where $L_1$ is the differential operator in $L^2(0,\infty)$ defined first by
\[
\big(L_1 f\big)(s):=-f''(s)+\bigg(\dfrac{q^2(d-1)^2-2q(d-1)}{4s^2}-\dfrac{1}{s^q}\bigg)f(s),
\quad f\in C^\infty_c(0,\infty),
\]	
and then extended using the Friedrichs extension.	
\end{theorem}

The result on the individual eigenvalues is then used to show that the corresponding eigenfunctions are localized
near the peak tip, which is rigorously described using an Agmon-type estimate as follows:
\begin{theorem}\label{thmeig}
	Let the assumptions of Theorem~\ref{thmain} be fulfilled. For any $j\in\N$ and $b>0$ there exist $B>0$ and $\alpha_0>0$ such that
	for any $\alpha>\alpha_0$ and any eigenfunction $u$ of $R_\alpha^\Omega$ for the eigenvalue $\lambda_j(R_\alpha^\Omega)$
	one has
	\begin{equation}
		\label{agm0}
		\int_{\Omega}e^{2b\alpha|x|}\big|u(x)\big|^2\, \dd \H^d(x)\le B\|u\|^{2}_{L^2(\Omega)}.
	\end{equation}
\end{theorem}

Let us make some comments on the assumptions and the relations with earlier works.
All eigenvalues $\lambda_j(L_1)$ are negative, see Subsection~\ref{sec1d} and Appendix~\ref{appa} below, so the eigenvalues $\lambda_j(R^\Omega_\alpha)$ diverge to $-\infty$ much faster than in the Lipschitz case. We are not aware of any value of $q\in(1,2)$ for which the eigenvalues of $L_1$ can be computed explicitly.
The assumption $q<2$ is necessary to ensure that the Sobolev trace operator is well-defined and compact from $W^{1,2}(\Omega)$ to $L^2(\partial\Omega)$, which guarantees the required properties of the above bilinear form $r^\Omega_\alpha$ and the self-adjointness of the Robin Laplacian $R^\Omega_\alpha$ for all $\alpha>0$, see \cite[Sec.~5.2]{daners} or \cite{acosta,nt}.
At the same time, the condition $q>1$ is required to actually have a peak, as $\Omega$ would be covered by the analysis of Lipschitz domains for $q\le 1$. The case when $\omega$ is a ball of radius $\rho>0$ was already studied in the earlier paper~\cite{kov}, and its main result is recovered (with an improved remainder estimate) by using \eqref{ljj} and observing that in this case
\[
\dfrac{\H^{d-2}(\partial\omega)}{\H^{d-1}(\omega)}=\dfrac{d-1}{\rho}.
\]
It should be noted that the analysis in \cite{kov} was crucially depending on a separation of variables in the ball (and lengthy manipulations with asymptotic expansions of special functions), which is indeed unavailable for general $\omega$. The central observation in this paper is that a part of the arguments of \cite{kov} based on a separation of variables can be replaced by an argument based on the first-order perturbation theory of linear operators and on an adapted coordinate change, which are summarized in Lemma \ref{basic properties} and Section \ref{ss-cyl}. This approach was mainly motivated by observations and geometric constructions from the
later paper~\cite{vogel}. The asymptotic expansion \eqref{ljj} reflects both the sharpness $q$ (through the order in $\alpha$) and the geometry of the cross-section $\omega$ (through the coefficient in the main term), which is an essentially new contribution when compared to~\cite{kov}.
This allows to make additional observations on the interplay between the geometry and the eigenvalues. For example, if the area or the perimeter of $\omega$ is fixed, then the ratio
$\H^{d-2}(\partial\omega)/\H^{d-1}(\omega)$
appearing in the main term of the asymptotic expansion is minimized by the balls due to the isoperimetric inequality. Therefore, if the peak cross-section is a ball, then for $\alpha\to+\infty$ the individual eigenvalues of $R^\Omega_\alpha$ diverge to $-\infty$ slower than for any other peak cross-section having the same area or the same perimeter, which is reminiscent of the famous Faber-Krahn inequality:

\begin{corollary}
	Let $\Omega\subset\R^{n}$ be as in Definition \ref{defin1} with cross-section $\omega$, and denote by $\Omega_0\subset\R^{n}$ a domain that is covered by Definition \ref{defin1} and has a ball $\omega_0$ as a cross-section. If $\H^{d-1}(\omega)=\H^{d-1}(\omega_0)$ or $\H^{d-2}(\partial\omega)=\H^{d-2}(\partial\omega_0)$, then for each $j\in\N$ there exists $\alpha_j>0$ such that
	\[
	\lambda_j(R_{\Omega}^\alpha)\le \lambda_j(R_{\Omega_0}^\alpha)\quad \text{ for all $\alpha>\alpha_j$}.
	\]
\end{corollary}

\begin{remark}	While it is not directly related to the scope of the present paper, we mention the fact
	that the Robin Laplacians show a rather surprising behavior when it comes to Faber-Krahn type inequalities. On one hand,
	among the bounded Lipschitz domains $\Omega$ of fixed volume and for any $\alpha\le0$ (in which case $R^\Omega_\alpha$ is a positive operator), the first eigenvalue $\lambda_1(R^\Omega_\alpha)$ is minimized by the balls, see \cite{daners2}. It was then conjectured in \cite{bareket,brock} that a reverse Faber-Krahn inequality holds true for any fixed $\alpha<0$, i.e.
	that the first eigenvalue $\lambda_1(R^\Omega_\alpha)$ is maximized by the balls among the bounded Lipschitz domains $\Omega$ of fixed volume. However, the conjecture was disproved in \cite{FK}, and the spectral asymptotics played a central role
	in the construction of a counterexample. This delivered a rare example of a spectral isoperimetric problem
	for which the balls are not optimizers.
\end{remark}

Our analysis can be easily adapted to the case of domains with several peaks: One considers smalls neighborhoods of the peak and
analyze them separately, and each peak then produces a series of eigenvalues as described in Theorem \ref{thmain}. In particular, the asymptotics of the individual eigenvalues is then determined by the sharpest peak(s), and the respective eigenfunctions are also localized near the respective peak tips. On the other hand, if there are several peaks
of the same sharpness, then tunneling effects cannot be excluded (i.e. eigenfunctions can localize simultaneously near several peaks), but their analysis is beyond the scope of the present work. The study of Robin eigenvalues can be addressed, in principle, for more general peaks. First, one can extend the class of possible $\Omega$ by replacing the condition $x'\in x_1^q\omega$ with the more general one 
 $x'\in \varphi(x_1)\omega$, where $\varphi$ is a strictly increasing smooth function with $\varphi(0)=\varphi'(0)=0$. While some first steps of the analysis are still applicable, the absence of principal homogeneous terms in various intermediate operators poses severe problems for the description of the eigenvalue asymptotics, and no power-type asymptotics in $\alpha$ can be expected. A further possible generalization is admitting non-Lipschitz cross-sections $\omega$ (in particular, those having peaks in a suitable defined sense). This case of ``iterated peaks'' is likely to give rise to a multiscale analysis of Laplacians in several dimensions, which is expected to be at a much higher complexity level than the present work. Another extension arises if one admits so-called non-isotropic peaks
featuring different scalings in different $x'$-directions. In this case a multi-step approach seems promising, and a particular case could recently be analyzed in \cite{vogel2}. As a further potential development we mention the possibility to replace the Laplacian through a magnetic Laplacian and to study the influence of magnetic fields: For smooth domains, there are some interesting spectral effects produced by competing Robin and magnetic contributions, see e.g. \cite{fahs, HKR, kach} and references therein, but we are not aware of any works dealing with similar questions in the presence of peaks.

The overall structure of the paper and of the proof follows closely the one in \cite{kov}. Section \ref{secprel} is devoted to technical preparations. In Subsection~\ref{secrobin} we collect important facts on Robin Laplacians proved in previous works. In Subsection \ref{sec1d} we introduce a family of one-dimensional operators, which includes the operator $L_1$ appearing in Theorem \ref{thmain}, and recall their  basic spectral and asymptotic properties. Subsection \ref{secpeak} introduces
a family of model peak domains together with associated Laplace-type operators. In Section~\ref{ss-cyl}
a coordinate change is employed to map the peak domains diffeomorphically to cylindrical domains
and to control various integral terms. This part is new with respect to \cite{kov}, and it is an adaptation of some computations from \cite{vogel}. In Subsection~\ref{nopeak}, we obtain a lower bound for model operators defined outside a neighborhood
of the peak's tip. The core of the study is Section~\ref{peak}, in which we study the eigenvalues
of a model operator defined in a small neighborhood of the peak's tip and relate them to the eigenvalues of $L_1$ as summarized in Corollary~\ref{corol14}. The proof avoids using a separation of variables
and employs the general results on Robin Laplacians from Section~\ref{secrobin} instead.
In Section~\ref{secproof} we make use of a series of truncations of $\Omega$ to show that only a small neighborhood of the peak's tip counts for the main term of the eigenvalue asymptotics, and an application of Corollary~\ref{corol14} completes the proof of Theorem~\ref{thmain}. In the last section (Section \ref{eigenfunctions}) we prove Theorem \ref{thmeig}
by using an IMS-type decoupling and showing that suitable integrals of the eigenfunctions are dominated by their contributions
near the peak tip.  Most part of the analysis is of purely variational nature with the help
of the min-max principle for eigenvalues through an adapted choice of test functions.

For the rest of the paper we assume that $\Omega\subset\R^d$ satisfies the assumptions of Theorem~\ref{thmain}. This fixes once and for all the sharpness parameter $q\in(1,2)$ and the cross-section $\omega\subset\R^{d-1}$, and we denote additionally
\[
A_\omega:=\dfrac{\H^{d-2}(\partial\omega)}{\H^{d-1}(\omega)}.
\]

	\section{Preparations}\label{secprel}
	\subsection{Robin Laplacians on bounded Lipschitz domains}\label{secrobin}

	The following proposition collects important properties of the eigenvalues of the Robin Laplacians, and we refer to \cite[Lemma 2.1]{vogel} for a proof.
	
	\begin{lemma}\label{basic properties}
		Let $U\subset\R^m$ be a bounded Lipschitz domain,
		then the following properties hold true for the Robin Laplacian $R^U_\alpha$ (as defined in the introduction):
		\begin{itemize}
			
			\item[(i)]Scaling: For any $t>0$, $\alpha\in\R$, $j\in\N$ one has
			\[
			\lambda_j(R^{tU}_\alpha)=\dfrac{\lambda_j(R^{U}_{t\alpha})}{t^2}.
			\]
			
			\item[(ii)] For any $\alpha\in\R$, the first eigenvalue $\lambda_1(R^{U}_{\alpha})$
				is simple and the corresponding eigenfunctions have constant sign in $U$.
				
			\item[(iii)] The function
			\[
			\R\ni\alpha\mapsto\lambda_1(R^U_\alpha)\in\R
			\]
			is smooth. If the associated eigenfunction $\psi_\alpha$ is chosen positive with $\|\psi_\alpha\|_{L^2(U)}=1$, then the mapping
			\[
			\R\ni\alpha\mapsto \psi_\alpha\in L^2(U)
			\]
			is also smooth.

			\item[(iv)]
			
			 There exists $\phi\in L^{\infty}(0,\infty)$ such that
\begin{gather*}
\lambda_1(R^{U}_{\alpha})=-\dfrac{\H^{m-1}(\partial U)}{\H^m(U)}\alpha+\alpha^2\phi(\alpha)\text{ for all }\alpha>0.
\end{gather*}
			
			\item[(v)] If $N_U$ is the Neumann Laplacian in $U$, then
			\[
			\lim_{\alpha\to0}\lambda_2(R^{U}_{\alpha})=\lambda_2(N_U)>0.
			\]

		\end{itemize}
	\end{lemma}
	
In what follows, Lemma \ref{basic properties} will be used mainly for $m=d-1$ and $U=\omega$.
	
	\subsection{A one-dimensional operator}\label{sec1d}
	For  $\mu>0$  consider the symmetric differential operator in $L^2(0,\infty)$ given by
	\begin{equation*}
		C_c^\infty(0,\infty) \ni f \mapsto -f'' + \left(\frac{q^2(d-1)^2-2q(d-1)}{4s^2}-\frac{\mu}{s^q}\right) f
	\end{equation*}
	and denote by $L_{\mu}$ its Friedrichs extension. Its essential spectrum is $[0,+\infty)$, and it has infinitely many negative eigenvalues due to the presence of the negative long-range potential $\mu/s^q$, and all the negative eigenvalues are simple. One easily checks that for the unitary scaling transformation
	\[
	Z_c:\ L^2(0,\infty)\to L^2(0,\infty), \quad Z_c f=\sqrt{c} f(c\,\cdot),\quad c>0,
	\]
	one has
	\begin{equation}
		\label{scaling1}
	L_\mu Z_c=c^2 Z_c L_{c^{q-2}\mu} \text{ for any $\mu>0$ and $c>0$},
	\end{equation}
	which shows that $L_\mu$ is unitarily equivalent to $c^2L_{c^{q-2}\mu}$. For $c:=\mu^{\frac{1}{2-q}}$ this implies the identity
	\begin{equation*}
		\label{scaling2}
	\lambda_j(L_\mu)=\mu^\frac{2}{2-q}\lambda_j(L_1) \text{ for any $\mu>0$ and $j\in\N$.}
	\end{equation*}
	
	In what follows we will deal with truncated versions of $L_{\mu}$. Namely, for $a>0$ we denote by $L_{\mu,a}$ the Friedrichs extension in $L^2(0,a)$ of the operator
	\[
	C_{c}^\infty (0,a)\ni f\mapsto L_{\mu} f.
	\]
	By construction the form domain of $L_{\mu,a}$ is continuously embedded in $W^{1,2}_0(0,a)$, which implies that $L_{\mu,a}$ has compact resolvent. In addition, the usual mollification procedure shows that any function from $W^{1,2}(0,a)$ vanishing in some neighborhoods of the endpoints belongs to the form domain of $L_{\mu,a}$ and, moreover, such functions build a core domain for the bilinear form of $L_{\mu,a}$. We will use the following asymptotic estimate for the eigenvalues of $L_{\mu,a}$:
	\begin{lemma}\label{modelop1}
		For any $a>0$ and $j\in\N$ there exist $K>0$ and $\mu_0>0$ such that
		\[
		\mu^\frac{2}{2-q}\lambda_j(L_1)\le
		\lambda_j(L_{\mu,a})\le \mu^{\frac{2}{2-q}} \lambda_j(L_{1})+K\quad \text{for all }\mu >\mu_0.
		\]
	\end{lemma}
	The proof is given in \cite[Sec.~3.1]{kov} for a slightly different choice of parameters, and it is translated into our language using the scaling \eqref{scaling1}.

\subsection{Finite peaks and related operators}\label{secpeak} For $\eps>0$ we will consider various finite pieces of the infinite peak
\[
V_\eps:=\Big\{(x_1,x')\in\R^d:\ x_1\in(0,\infty),\ x'\in \eps x_1^q\omega\Big\}.
\]
Namely, for an open interval $I\subset(0,\infty)$ 
it will be convenient to denote
\begin{align*}
	V_{\eps,I}&:= V_\eps \cap\Big( I\times\R^{d-1}\Big)
\equiv \Big\{(x_1,x')\in\R^d:\ x_1\in I,\ x'\in \eps x_1^q\omega\Big\},
\end{align*}
see Figure~\ref{fig2}.
In particular, one has $V_\eps=V_{\eps,(0,\infty)}$. We will also consider the ``lateral boundary'' $\partial_0 V_{\eps,I}$ of $V_{\eps,I}$ given by
\begin{align*}
	\partial_0V_{\eps,I}&:= \partial V_\eps \cap\Big( I\times\R^{d-1}\Big)
	\equiv \Big\{(x_1,x')\in\R^d:\ x_1\in I,\ x'\in \eps x_1^q\partial\omega\Big\}.
\end{align*}

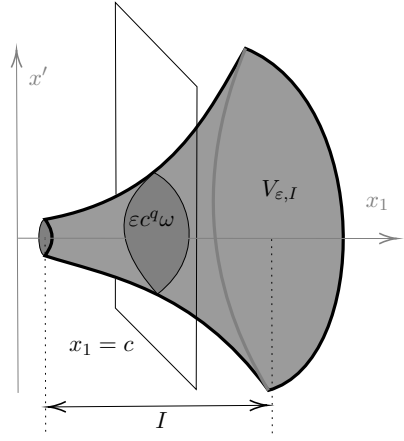
\begin{figure}
	\centering
	
	\scalebox{0.8}{
	\begin{tikzpicture}[x=0.75pt,y=0.75pt,yscale=-1,xscale=1]
		
		\draw [color={rgb, 255:red, 128; green, 128; blue, 128 }  ,draw opacity=1 ]   (154.6,262.2) -- (153.61,49.2) ;
		\draw [shift={(153.6,47.2)}, rotate = 89.73] [color={rgb, 255:red, 128; green, 128; blue, 128 }  ,draw opacity=1 ][line width=0.75]    (10.93,-3.29) .. controls (6.95,-1.4) and (3.31,-0.3) .. (0,0) .. controls (3.31,0.3) and (6.95,1.4) .. (10.93,3.29)   ;
		\draw  [fill={rgb, 255:red, 128; green, 128; blue, 128 }  ,fill opacity=1 ] (167.6,164.8) .. controls (167.6,158.5) and (169.35,153.4) .. (171.5,153.4) .. controls (173.65,153.4) and (175.4,158.5) .. (175.4,164.8) .. controls (175.4,171.1) and (173.65,176.2) .. (171.5,176.2) .. controls (169.35,176.2) and (167.6,171.1) .. (167.6,164.8) -- cycle ;
		\draw  [dash pattern={on 0.84pt off 2.51pt}]  (171.5,164.8) -- (171.9,287.8) ;
		\draw    (175.6,270.99) -- (308.8,270.01) ;
		\draw [shift={(310.8,270)}, rotate = 179.58] [color={rgb, 255:red, 0; green, 0; blue, 0 }  ][line width=0.75]    (10.93,-3.29) .. controls (6.95,-1.4) and (3.31,-0.3) .. (0,0) .. controls (3.31,0.3) and (6.95,1.4) .. (10.93,3.29)   ;
		\draw [shift={(173.6,271)}, rotate = 359.58] [color={rgb, 255:red, 0; green, 0; blue, 0 }  ][line width=0.75]    (10.93,-3.29) .. controls (6.95,-1.4) and (3.31,-0.3) .. (0,0) .. controls (3.31,0.3) and (6.95,1.4) .. (10.93,3.29)   ;
		\draw  [color={rgb, 255:red, 0; green, 0; blue, 0 }  ,draw opacity=1 ][fill={rgb, 255:red, 155; green, 155; blue, 155 }  ,fill opacity=1 ][line width=1.5]  (171.5,153.4) .. controls (250.4,139.4) and (274.4,85) .. (297.2,46.2) .. controls (377.2,79) and (375.6,242.6) .. (311.6,260.6) .. controls (264.8,196.2) and (210,188.6) .. (171.5,176.2) .. controls (174.8,171) and (179.6,164.6) .. (171.5,153.4) -- cycle ;
		\draw  [dash pattern={on 0.84pt off 2.51pt}]  (313.6,165.8) -- (314,288.8) ;
		\draw    (216,137.8) -- (215.6,17.4) -- (266.8,70.2) -- (266.3,260) -- (215.6,209) -- (215.6,190.6) ;
		\draw [color={rgb, 255:red, 128; green, 128; blue, 128 }  ,draw opacity=1 ][line width=1.5]    (311.6,260.6) .. controls (293.6,230.6) and (252.4,154.2) .. (297.2,46.2) ;
		\draw  [fill={rgb, 255:red, 128; green, 128; blue, 128 }  ,fill opacity=1 ] (240,124.2) .. controls (220.4,141) and (208.8,163) .. (242,200.2) .. controls (274.4,175.4) and (262.4,135.4) .. (240,124.2) -- cycle ;
		\draw [color={rgb, 255:red, 128; green, 128; blue, 128 }  ,draw opacity=1 ]   (154.6,165.2) -- (390,165.4) ;
		\draw [shift={(392,165.4)}, rotate = 180.05] [color={rgb, 255:red, 128; green, 128; blue, 128 }  ,draw opacity=1 ][line width=0.75]    (10.93,-3.29) .. controls (6.95,-1.4) and (3.31,-0.3) .. (0,0) .. controls (3.31,0.3) and (6.95,1.4) .. (10.93,3.29)   ;
		
		\draw (371.5,138) node [anchor=north west][inner sep=0.75pt]  [color={rgb, 255:red, 128; green, 128; blue, 128 }  ,opacity=1 ]  {$x_{1}$};
		\draw (160.1,56.6) node [anchor=north west][inner sep=0.75pt]  [color={rgb, 255:red, 128; green, 128; blue, 128 }  ,opacity=1 ]  {$x'$};
		\draw (185.3,228.6) node [anchor=north west][inner sep=0.75pt]    {$x_{1} =c$};
		\draw (239.6,273.2) node [anchor=north west][inner sep=0.75pt]    {$I$};
		\draw (222.8,146.6) node [anchor=north west][inner sep=0.75pt]    {$\varepsilon c^{q} \omega $};
		\draw (306,127.2) node [anchor=north west][inner sep=0.75pt]    {$V_{\varepsilon ,I}$};

	\end{tikzpicture}
	}
	
	\caption{An example of the model peak domain $V_{\eps,I}$.}\label{fig2}
	
\end{figure}

At several places we will work with functions localized in the first variable, so let us introduce
an adapted notation. For an open set $U\subset\R^d$ and an open interval $I\subset\R$ it will be convenient to denote
\begin{align*}
	W^{1,2}_I(U):=\Big\{& u\in W^{1,2}(U):\ \exists c,c'\in I \text{ such that }u(x)=0 \text{ for all $x\in U$ with $x_1\notin[c,c']$}\Big\}.
\end{align*}

For $\alpha\in\R$ and an open interval $I\subset(0,\infty)$ consider the symmetric bilinear form $t^{\alpha,N}_{\eps, I}$ given by
\begin{align*}
	t^{\alpha,N}_{\eps,I}(u,u)&:=\int_{V_{\eps,I}}|\nabla u|^2\dd\H^d-\alpha\int_{\partial_0 V_{\eps,I}} u^2\dd\H^{d-1},\qquad
	\cD(t^{\alpha,N}_{\eps,I}):=W^{1,2}(V_{\eps,I}).
\end{align*}

In the final steps of the proof of the main theorem we will use the following density result:
\begin{lemma}\label{lem5}
	For any open bounded interval $I\subset(0,\infty)$ and any $\eps>0$ the set $W^{1,2}_{(0,\infty)}(V_{\eps,I})$ is dense in $W^{1,2}(V_{\eps,I})$.
\end{lemma}

\begin{proof}
Remark that $W^{1,2}_{(0,\infty)}(V_{\eps,I})$ consists of the functions in $W^{1,2}(V_{\eps,I})$ that vanish in a neighborhood of the origin. If $0$ is not an endpoint of $I$, then $W^{1,2}_{(0,\infty)}(V_{\eps,I})=W^{1,2}(V_{\eps,I})$, and there is nothing to prove. Therefore, from now on we consider the case $I=(0,\ell)$ with $\ell\in(0,\infty)$.

From the general theory of Sobolev spaces it is known that
$W^{1,2}(V_{\eps,I})\cap L^\infty(V_{\eps,I})$ is a dense subspace of $W^{1,2}(V_{\eps,I})$,
see~\cite[Theorem in Sec.~1.4.3]{mazya}. Therefore, it is sufficient to show that any $u\in W^{1,2}(V_{\eps,I})\cap L^\infty(V_{\eps,I})$ can be approximated by functions from $W^{1,2}_{(0,\infty)}(V_{\eps,I})$.
Pick a function $\chi\in C^\infty(0,\infty)$ such that
\[
0\le \chi\le 1,\qquad 
\chi(s)=0 \text{ for }s<1,\qquad \chi(s)=1 \text{ for }s>2,
\]
and for small $\mu>0$ consider the functions
\[
u_\mu:\ (x_1,x')\mapsto \chi\Big(\frac{x_1}{\mu}\Big)u(x_1,x').
\]
By construction $u_{\mu}\in W^{1,2}(V_{\eps,I})$ with $u_{\mu}(x_1,x')=0$ for $x_1<\mu$, so $u_{\mu}\in W^{1,2}_{(0,\infty)}(V_{\eps,I})$. We are going to show that $u_\mu$ converges to $u$ in
$W^{1,2}(V_{\eps,I})$ for $\mu\to 0^+$.

Using the dominated convergence theorem one shows that
$u_\mu$ converges to $u$ in $L^2(V_{\eps,I})$ for $\mu\to 0^+$.
For each $j\ge 2$ we have
\[
\partial_j u_\mu:\ (x_1,x')\mapsto \chi\Big(\frac{x_1}{\mu}\Big)\partial_j u(x_1,x'),
\]
and the same argument shows the convergence of $\partial_j u_\mu$ to $\partial_j u$ in $L^2(V_{\eps,I})$ for $\mu\to 0^+$. Furthermore,
\[
\partial_1 u_\mu:\ (x_1,x')\mapsto \frac{1}{\mu}\chi'\Big(\frac{x_1}{\mu}\Big)u(x_1,x')
+ \chi\Big(\frac{x_1}{\mu}\Big)\partial_1 u(x_1,x'),
\]
in particular, $\partial_1 u_\mu(x_1,x')=\partial_1 u(x_1,x')$ for $x_1>2\mu$, and
\begin{align*}
\|\partial_1 u_\mu&-\partial_1 u\|^2_{L^2(V_{\eps,I})}=	\|\partial_1 u_\mu-\partial_1 u\|^2_{L^2(V_{\eps,(0,2\mu)})}\\
&=\int_0^{2\mu} \int_{\eps x_1^q \omega}\bigg(\frac{1}{\mu}\chi'\Big(\frac{x_1}{\mu}\Big)u(x_1,x')
+ \Big[\chi\Big(\frac{x_1}{\mu}\Big)-1\Big]\partial_1 u(x_1,x')\bigg)^2\dd \H^{d-1}(x')\,\dd x_1\\
&\le I'_\mu+I''_\mu,\\
I'_\mu&:=\dfrac{2}{\mu^2}\int_0^{2\mu} \int_{\eps x_1^q \omega} \Big[\chi'\Big(\frac{x_1}{\mu}\Big)u(x_1,x')\Big]^2\dd \H^{d-1}(x')\,\dd x_1,\\
I''_\mu&:=2\int_0^{2\mu} \int_{\eps x_1^q \omega}
\bigg(\Big[\chi\Big(\frac{x_1}{\mu}\Big)-1\Big]\partial_1 u(x_1,x')\bigg)^2\dd \H^{d-1}(x')\,\dd x_1.
\end{align*}
Using $q(d-1)>1$ we obtain
\begin{align*}
I'_\mu&\le \dfrac{2}{\mu^2}\int_0^{2\mu} \int_{\eps x_1^q \omega} \|\chi'\|_\infty^2 \|u\|_\infty^2\dd \H^{d-1}(x')\,\dd x_1\\
&=\dfrac{2 \|\chi'\|_\infty^2 \|u\|_\infty^2}{\mu^2} \int_0^{2\mu} \H^{d-1}(\eps x_1^q \omega) \,\dd x_1\\
&=\dfrac{2 \|\chi'\|_\infty^2 \|u\|_\infty^2}{\mu^2} \int_0^{2\mu} (\eps x_1^q)^{d-1}\H^{d-1}(\omega)\,\dd x_1\\
&=\dfrac{2 \|\chi'\|_\infty^2 \|u\|_\infty^2\eps^{d-1} \H^{d-1}(\omega)}{\mu^2} \int_0^{2\mu} x_1^{q(d-1)}\dd x_1\\
&=\dfrac{2 \|\chi'\|_\infty^2 \|u\|_\infty^2\eps^{d-1} \H^{d-1}(\omega)}{\mu^2} \dfrac{(2\mu)^{q(d-1)+1}}{q(d-1)+1}\\
&=\dfrac{2^{q(d-1)+2} \|\chi'\|_\infty^2 \|u\|_\infty^2\eps^{d-1} \H^{d-1}(\omega)}{q(d-1)+1} \mu^{q(d-1)-1}\xrightarrow{\mu\to 0^+}0,
\end{align*}
while $I''_\mu$ converges to $0$ for $\mu\to 0^+$ due to the dominated convergence theorem,
which implies the convergence of $\partial_1 u_\mu$ to $\partial_1 u$ in $L^2(V_{\eps,I})$ for $\mu\to 0^+$.
\end{proof}

Denote
\[
\widehat{W}_0^{1,2}(V_{\eps,I}):=\text{the closure of }W^{1,2}_I(V_{\eps,I})\text{ in }W^{1,2}(V_{\eps,I})
\]
and consider the symmetric bilinear form
\[
t^{\alpha,D}_{\eps, I}:=\text{the restriction of  $t^{\alpha,N}_{\eps, I}$ on $\widehat{W}_0^{1,2}(V_{\eps,I})$.}
\]
By construction we have:
\begin{equation}
\begin{aligned}
	W^{1,2}_I(V_{\eps,I})&\text{ is a core domain of }t^{\alpha,D}_{\eps,I},
\end{aligned}
 \label{eq-core}
\end{equation}
which will simplify the subsequent considerations.

In what follows we will be interested in the spectral analysis of the self-adjoint operators $T^{\alpha,N/D}_{\eps,I}$
acting in $L^2(V_{\eps,I})$ and defined by the forms $t^{\alpha,N/D}_{\eps,I}$. Note that for any $\alpha>0$ and any $I$ one has
\begin{align*}
		\alpha V_{1,I}&=\Big\{ (x_1,x')\in\R^d:\ \big(\dfrac{x_1}{\alpha},\dfrac{x'}{\alpha}\big)\in V_{1,I}\Big\}\\
		&=\Big\{ (x_1,x')\in\R^d:\ \dfrac{x_1}{\alpha}\in I,\ \dfrac{x'}{\alpha}\in \big(\dfrac{x_1}{\alpha}\big)^q \omega\Big\}\\
		&=\Big\{ (x_1,x')\in\R^d:\ x_1\in \alpha I,\ x'\in \alpha^{1-q}x_1^q \omega\Big\}\equiv V_{\alpha^{1-q},\alpha I}.
\end{align*}
Therefore, a simple scaling argument shows that for any $j\in\N$ one has the relations
\begin{equation}
	\label{lscal}
\lambda_j(T^{\alpha,N/D}_{1,I})=\alpha^2 \lambda_j(T^{1,N/D}_{\alpha^{1-q},\alpha I}).
\end{equation}
Our next goal is to obtain lower and upper bounds for the eigenvalues of $T^{\alpha,N/D}_{\eps, I}$
with the help of suitable coordinate changes.

\subsection{Reduction to cylindrical domains}\label{ss-cyl}
For $\eps>0$ consider the diffeomorphism
\[
F_\eps:(0,\infty)\times\R^{d-1}\to(0,\infty)\times\R^{d-1},\qquad F_\eps(s,t):=(s,\eps s^qt),
\]
then for any open interval $I\subset(0,\infty)$ we have	
\[
F_\eps(\Pi_I)=V_{\eps,I} \quad \text{for}\quad
\Pi_I:=I\times\omega.
\]
In order to deal with various integrals under the change of variables defined by $F_\eps$,
we will perform some preliminary computations. Let us introduce the constant
\[
R_\omega:=\sup_{t\in\omega}|t|.
\]

\begin{lemma}\label{two side est of bdy integral}
	For any $\eps>0$, any measurable function $v:\partial_0 V_{\eps,I}\to\R$ and $u:=v\circ F_\eps$ it holds
	\begin{align*}\label{bdy integral est0}
		\eps^{d-2}\int_I&s^{q(d-2)}\int_{\partial\omega}\big|u(s,t)\big|\,\dd\H^{d-2}(t)\,\dd s\leq
		\int_{\partial_0 V_{\eps,I}}|v|\,\dd\H^{d-1}\\
		&\leq\eps^{d-2}\int_I
		s^{q(d-2)}\sqrt{1+\eps^2q^2\,R_\omega^2\,s^{2q-2}}\int_{\partial\omega}\big|u(s,t)\big|\,\dd\H^{d-2}(t)\,\dd s\notag.
	\end{align*}
\end{lemma}
\begin{proof}
	It is sufficient to prove the result for the functions $v$ supported in images of  local charts, then it is extended to general functions by using a partition of unity. Let
	\[
	U\ni z=(z_1,z_2,\cdots,z_{d-2})\mapsto\psi(z)\in \partial\omega
	\]
	be a local chart on $\partial\omega$, then $U\ni z\mapsto\eps\psi(z)\in \partial(\eps\omega)$
	is a local chart on $\partial(\eps\omega)$, and
	\[
	\Psi_\eps:\ I\times U\ni(s,z)\mapsto\big(s, \eps s^q \psi(z)\big)\equiv F_\eps\big(s,\psi(z)\big)\in\partial_0 V_{\eps,I}
	\]
	is a local chart on $\partial_0 V_{\eps,I}$. If $v$ is supported in the image of $\Psi_\eps$, then the expression of $\H^{d-1}$ on hypersurfaces gives
	\begin{equation}\label{bdy integral est}
		\int_{\partial_0 V_{\eps,I}}|v|\,\dd\H^{d-1}=\int_I\int_{U}\Big|v\big(\Psi_\eps(s,z)\big)\Big|J_\eps(s,z)\,\dd\H^{d-2}(z)\dd s
	\end{equation}
	with
	\[
	J_\eps(s,z):=\sqrt{\det\big(D\Psi_\eps(s,z)^TD\Psi_\eps(s,z)\big)}.
	\]
	A direct computation gives
	\[
	D\Psi_\eps(s,z)=\begin{pmatrix}
		1 & 0 & \dots & 0\\
		\eps q s^{q-1}\psi(z) & \eps s^q \partial_1 \psi(z) & \dots & \eps s^q \partial_{d-2}\psi(z)
	\end{pmatrix},
	\]
	which yields
	\[
	D\Psi_\eps(s,z)^TD\Psi_\eps(s,z)=
	\begin{pmatrix}
		1+\eps^2q^2s^{2q-2}\big|\psi(z)\big|^2 & \eps^2qs^{2q-1}H(z)^T\,\\
		\\
		\eps^2qs^{2q-1}H(z) & \eps^2s^{2q}G(z)
	\end{pmatrix}
	\]
	with $G(z):=D\psi(z)^T D\psi(z)$ and
	\[		
	H(z):=
	\begin{pmatrix}
		\big\langle\psi(z), \partial_1\psi(z)\big\rangle\\
		\vdots\\
		\big\langle\psi(z), \partial_{d-2}\psi(z)\big\rangle\\
	\end{pmatrix}
	\equiv\Big(|\psi|\nabla|\psi|\Big)(z).
	\]
	
	Since $\psi$ is a local chart, the matrix $G (z)$ is invertible for a.e. $z$, and a standard formula for computing the determinant of a block matrix with an invertible diagonal block
	becomes applicable, which leads to
	\begin{align*}
		J_\eps(s,z)^2&=\eps^{2(d-2)}s^{2q(d-2)}\det G(z)\Big[ 1+\eps^2q^2s^{2q-2}\big|\psi(z)\big|^2
		-\eps^2q^2s^{2q-2} H(z)^T G(z)^{-1} H(z)\Big]\\
		&=\eps^{2(d-2)}s^{2q(d-2)}\det G(z)\Big[
		1+\eps^2q^2s^{2q-2}\big|\psi(z)\big|^2\Big( 1- \Big\langle \nabla|\psi|(z), G(z)^{-1} \nabla|\psi|(z)\Big\rangle\Big)\Big].
	\end{align*}
	Consider the modulus function
	\[
	h:\ \partial\omega\ni t\mapsto |t|\in \R,
	\]
	then
	\begin{gather*}
		\Big\langle \nabla|\psi|(z), G(z)^{-1} \nabla|\psi|(z)\Big\rangle=|\nabla^{\partial\omega}h|^2\big(\psi(z)\big),\\
		\nabla^{\partial\omega}h:=\text{the tangential gradient of $h$ along $\partial\omega$,}
	\end{gather*}
	and
	\begin{align*}
		1- \Big\langle \nabla|\psi|(z), G(z)^{-1} \nabla|\psi|(z)\Big\rangle&=
		|\nabla h|^2\big(\psi(z)\big)-|\nabla^{\partial\omega}h|^2\big(\psi(z)\big)=|\partial_{\nu}h|^2\big(\psi(z)\big),
	\end{align*}
	where $\partial_\nu$ stands for the normal derivative on $\partial\omega$. 
	Therefore,
	\[
	J_\eps(s,z)^2=\eps^{2(d-2)}s^{2q(d-2)}\det G(z)\Big[
	1+\eps^2q^2s^{2q-2}\big|\psi(z)\big|^2 |\partial_{\nu}h|^2\big(\psi(z)\big)\Big],
	\]
	and one arrives at the obvious lower bound
	\[
	J_\eps(s,z)^2\ge \eps^{2(d-2)}s^{2q(d-2)}\det G(z).
	\]
	Furthermore, due to the definition of $h$ we have $|\partial_\nu h|\le 1$ a.e,
	which implies
	\begin{align*}
		J_\eps(s,z)^2&\le \eps^{2(d-2)}s^{2q(d-2)}\det G(z)\Big(
		1+\eps^2q^2s^{2q-2}\big|\psi(z)\big|^2\Big)\\
		&\le
		\eps^{2(d-2)}s^{2q(d-2)}\det G(z) \Big(
		1+\eps^2q^2s^{2q-2} R_\omega^2\Big).
	\end{align*}
	By substituting these bounds for $J_\eps(s,z)^2$ into \eqref{bdy integral est}
	we arrive at
	\begin{align*}
		\eps^{d-2}&\int_I\int_U s^{q(d-2)}\Big|v\big(\Psi_{\epsilon}(s,z)\big)\Big|\sqrt{\det G(z)}
		\,\dd\H^{d-2}(z)\,\dd s			
		\leq\int_{\partial_0 V_{\eps,I}}|v|\dd\mathcal{H}^{d-1}
		\\
		&\leq\eps^{d-2}\int_I\int_U s^{q(d-2)}\sqrt{1+\eps^2q^2s^{2q-2}R^2_\omega}\Big|v\big(\Psi_{\epsilon}(s,z)\big)\Big|\sqrt{\det G(z)} \,\dd\H^{d-2}(z)\,\dd s. 
	\end{align*}
It remains to note that $v\big(\Psi_{\epsilon}(s,z)\big)=u(s,\psi(z))$ and
\[
\int_U \big|u(s,\psi(z))\big|\sqrt{\det G(z)} \,\dd\H^{d-2}(z)=\int_{\partial\omega} \big|u(s,\cdot)\big|\dd\H^{d-2}. \qedhere
\]
\end{proof}

\begin{lemma}\label{two side est of gradient}
	For any $v\in W^{1,2}(V_{\eps,I})$ and $u:=v\circ F_\eps$ we have
\begin{align*}
			\eps^{d-1}\int_Is^{q(d-1)}&\int_{\omega}\Bigg[\big(1-(d-1)\eps qR_\omega\big)|\partial_su|^2\\
			&+\left(\dfrac{1}{\eps^2s^{2q}}-\frac{qR_\omega}{s^2\eps}-(d-1)\frac{q^2R^2_\omega}{s^2}\right)|\nabla_tu|^2\Bigg]\,\dd\H^{d-1}(t)\,\dd s\\
			&\leq\int_{V_{\eps,I}}|\nabla v|^2\dd\H^d\\
			&\leq\eps^{d-1}\int_I s^{q(d-1)}\int_{\omega}\Bigg[\left(1+(d-1)\eps qR_\omega\right)|\partial_su|^2\\
			&\qquad+\left(\dfrac{1}{\eps^2s^{2q}}+\frac{qR_\omega}{s^2\eps}+(d-1)\frac{q^2R^2_\omega}{s^2}\right)|\nabla_tu|^2\Bigg]\,\dd\H^{d-1}(t)\,\dd s.
\end{align*}
\end{lemma}
\begin{proof}
Remark that
\begin{gather*}
DF_\eps(s,t)=\begin{pmatrix}
	1 & 0\\
	\eps q s^{q-1} t & \eps s^q I_{d-1}
\end{pmatrix},
\quad 
t=\begin{pmatrix}
	t_1\\
	\dots\\
	t_{d-1}
\end{pmatrix},\\
\quad
I_{d-1}:=\text{the $(d-1)\times (d-1)$ identity matrix.}
\end{gather*}
The change of variables $x=F_\eps(s,t)$ gives
\[
		\int_{V_{\eps,I}}\big|\nabla v(x)\big|^2\dd\H^d(x)=\int_I\int_{\omega}\big\langle\nabla u,G_\eps\nabla u\big\rangle_{\R^d} \,g_\eps \,\dd\H^{d-1}(t)\,\dd s,
\]
where
\begin{align*}
g_\eps(s,t)&:=\Big|\det DF_\eps(s,z)\Big|\equiv \eps^{d-1}s^{q(d-1)},\\
G_\eps(s,t)&:=\Big(DF_\eps(s,t)^{T}DF_\eps(s,t)\Big)^{-1}\\
&=
		\begin{pmatrix}
			1+\eps^2q^2\,s^{2(q-1)}|t|^2 & \eps^2 q s^{2q-1}t^T\\
			\eps^2 qs^{2q-1}t & \eps^2 s^{2q}I_{d-1}
		\end{pmatrix}^{-1}
=
		\begin{pmatrix}
			1 & -\dfrac{q}{s}t^T\\
			\\
			-\dfrac{q}{s}t  & \dfrac{1}{\eps^2s^{2q}}I_{d-1}+ \dfrac{q^2}{s^2}\, t t^T
		\end{pmatrix}.
\end{align*}
It follows that
	\begin{equation}\label{est of inner product}
		\begin{aligned}
		\langle\nabla u,G_\eps\nabla u\rangle_{\R^d}&=|\partial_su|^2+\dfrac{1}{\eps^2s^{2q}}|\nabla_tu|^2\\
		&\quad-\frac{2q}{s}\sum_{j=1}^{d-1}t_j\partial_su\,\partial_{t_j}u+\frac{q^2}{s^2}\sum_{j,\,k=1}^{d-1}t_jt_k\partial_{t_j}u\,\partial_{t_k}u.
		\end{aligned}
	\end{equation}
We estimate	
\begin{equation}\label{est of inner product 1term}
\begin{aligned}
\left|\frac{2q}{s}\sum_{j=1}^{d-1}t_j\partial_su\,\partial_{t_j}u\right|&
\le q\sum_{j=1}^{d-1}|t_j|\,2\Big|\partial_su\,\dfrac{\partial_{t_j}}{s}u\Big|
\le qR_\omega \sum_{j=1}^{d-1} 2\Big|\partial_su\,\dfrac{\partial_{t_j}u}{s}\Big|\\
&			\leq q R_\omega\sum_{j=1}^{d-1}\bigg[\eps |\partial_su|^2 + \dfrac{1}{\eps}\Big|\dfrac{\partial_{t_j}u}{s}\Big|^2\bigg]
			=(d-1)q R_\omega \eps |\partial_su|^2+\frac{qR_\omega}{\eps s^2}|\nabla_tu|^2.
		\end{aligned}
\end{equation}
Similarly, we have
	\begin{equation}\label{est of inner product 2term}
		\begin{aligned}
		\left|\frac{q^2}{s^2}\sum_{j,\,k=1}^{d-1}t_jt_k\partial_{t_j}u\,\partial_{t_k}u\right|
		&\le\frac{q^2}{s^2} \sum_{j,\,k=1}^{d-1}|t_j|\, |t_k|\, |\partial_{t_j}u|\,|\partial_{t_k}u|\\
		&\le \frac{q^2R_\omega^2}{s^2} \sum_{j,\,k=1}^{d-1}\, |\partial_{t_j}u|\,|\partial_{t_k}u|\\
		&\le \frac{q^2R_\omega^2}{2s^2} \sum_{j,\,k=1}^{d-1}\, \Big(|\partial_{t_j}u|^2+|\partial_{t_k}u|^2\Big)\\
		&=\frac{q^2R^2_\omega}{s^2}(d-1)|\nabla_tu|^2.
		\end{aligned}
	\end{equation}
By using \eqref{est of inner product 1term} and \eqref{est of inner product 2term} to estimate
the two last summands on the right-hand side of \eqref{est of inner product} we obtain the desired inequality.
\end{proof}

Note that the above computations are classical for smooth $\partial\omega$, but the formulas are still
valid for Lipschitz $\omega$ due to general results on Lipschitz manifolds \cite{rosen}.

\subsection{A lower bound outside the peak}\label{nopeak} The goal of this subsection is to obtain a lower bound
for the eigenvalues of $T^{1,N}_{\eps,I}$ for intervals $I$ separated from $0$: It will be used in the truncation arguments in the next sections.

\begin{lemma}\label{lem8}
Let $b,B>0$ and $\eps_0>0$. Then there exists a constant $c>0$ such that for any $\eps\in(0,\eps_0)$
and any non-empty open interval
\[
I_\eps\subset \big(b,B\eps^{-\frac{1}{q-1}}\big)
\]
it holds $\lambda_1(T^{1,N}_{\eps,I_\eps})\ge -c\eps^{-1}$.
\end{lemma}

\begin{proof}
For any $u\in\cD(T^{1,N}_{\eps,I_\eps})$ we have by Fubini's theorem
\begin{align*}
	\int_{V_{\eps,I_\eps}} |\nabla u|^2\dd \H^d&\ge \int_{V_{\eps,I_\eps}} |\nabla_{x'} u|^2\dd \H^d=\int_{I_\eps} \int_{\eps x_1^q \omega} \big|\nabla_{x'}u(x_1,x')\big|^2\dd \H^{d-1}(x')\,\dd x_1.
\end{align*}	
In addition, due to Lemma \ref{two side est of bdy integral},	
\begin{align*}
	\int_{\partial_0 V_{\eps,I_\eps}}u^2\dd\H^{d-1}
	&\le \eps^{d-2}\int_{I_\eps}
	x_1^{q(d-2)}\sqrt{1+\eps^2q^2\,R_\omega^2\,x_1^{2q-2}}\int_{\partial\omega}u(x_1,\eps x_1^q x')^2\,\dd\H^{d-2}(x')\,\dd x_1\\
	&\le \int_{I_\eps}\sqrt{1+\eps^2q^2\,R_\omega^2\,x_1^{2q-2}}
	\int_{\partial(\eps x_1^q\omega)}u(x_1,x')^2\,\dd\H^{d-2}(x')\,\dd x_1.
\end{align*}
Therefore,
\begin{align*}
	t^{1,N}_{\eps,I_\eps}(u,u)&=\int_{V_{\eps,I_\eps}} |\nabla u|^2\dd \H^d-\int_{\partial_0 V_{\eps,I_\eps}}u^2\dd\H^{d-1}\\
	&\ge \int_{I_\eps}\Big[ \int_{\eps x_1^q \omega} \big|\nabla_{x'}u(x_1,x')\big|^2\dd\,\H^{d-1}(x')\\
	&\qquad-
	\sqrt{1+\eps^2q^2\,R_\omega^2\,x_1^{2q-2}}
	\int_{\partial(\eps x_1^q\omega)}u(x_1,x')^2\,\dd\H^{d-2}(x')\Big]\,\dd x_1\\
	&=\int_{I_\eps} r^{\eps x_1^q \omega}_{\sqrt{1+\eps^2q^2\,R_\omega^2\,x_1^{2q-2}}}\big(u(x_1,\cdot),u(x_1,\cdot)\big)\,\dd x_1\\
	&\ge \int_{I_\eps} \lambda_1\Big(R^{\eps x_1^q \omega}_{\sqrt{1+\eps^2q^2\,R_\omega^2\,x_1^{2q-2}}}\Big) \int_{\eps x_1^q \omega}u(x_1,x')^2\,\dd\H^{d-1}(x')\,\dd x_1\\
	&\ge \Lambda_\eps \int_{V_{\eps,I_\eps}}u^2\dd\H^d,\qquad
\Lambda_\eps:=\inf_{x_1\in I_\eps} \lambda_1\left(R^{\eps x_1^q \omega}_{\sqrt{1+\eps^2q^2\,R_\omega^2\,x_1^{2q-2}}}\right),
\end{align*}
i.e. $\lambda_1(T^{1,N}_{\eps,I_\eps})\ge \Lambda_\eps$ for all $\eps\in(0,\eps_0)$,
and it remains to find a suitable lower bound for $\Lambda_\eps$.
With the help of the assertions (i) and (iv) of Lemma \ref{basic properties} we obtain, with some $C'>0$,
\begin{align*}
\lambda_1\left(R^{\eps x_1^q \omega}_{\sqrt{1+\eps^2q^2\,R_\omega^2\,x_1^{2q-2}}}\right)&=\dfrac{\lambda_1\left(R^{\omega}_{\eps x_1^q\sqrt{1+\eps^2q^2\,R_\omega^2\,x_1^{2q-2}}}\right)}{\eps^2x_1^{2q}}\\
&\ge -\dfrac{
	\eps A_\omega  x_1^q\sqrt{1+\eps^2q^2\,R_\omega^2\,x_1^{2q-2}}
	+C' \eps^2 x_1^{2q} \big(1+\eps^2q^2\,R_\omega^2\,x_1^{2q-2}\big)
	}{\eps^2 x_1^{2q}}\\
	&= -A_\omega \dfrac{\sqrt{1+\eps^2q^2\,R_\omega^2\,x_1^{2q-2}}}{\eps x_1^q}
	-C' \big(1+\eps^2q^2\,R_\omega^2\,x_1^{2q-2}\big).
\end{align*}

For all $x_1\in I_\eps$ we have
\[
 x_1^q>b^q,\qquad
\eps x_1^{q-1}< \eps \big(B \eps^{-\frac{1}{q-1}}\big)^{q-1}=B^{q-1},
\]
therefore, for any $\eps\in(0,\eps_0)$ one has
\begin{align*}
\Lambda_\eps &
\ge -\dfrac{A_\omega\sqrt{1+q^2R_\omega B^{2q-2}}}{b^q}\cdot\dfrac{1}{\eps}-C'(1+q^2R_\omega^2 B^{2q-2})\\
&\ge -\bigg[ \dfrac{A_\omega\sqrt{1+q^2R_\omega B^{2q-2}}}{b^q} + C'(1+q^2R_\omega^2 B^{2q-2}) \eps_0\bigg]\cdot\dfrac{1}{\eps}. \qedhere
\end{align*}
\end{proof}

\section{Spectral asymptotics near the peak}\label{peak}

\subsection{Model operators near the peak}
We are going to apply a series of coordinate changes, which will allow for the spectral study of the operators
\begin{equation}
	\label{tea}
T_{\eps,a}:=T^{1,D}_{\eps,(0,a)}
\end{equation}
as $\eps\to 0^+$. We also denote
\[
t_{\eps,a}:=t^{1,D}_{\eps,(0,a)}, \qquad \Pi_a:=\Pi_{(0,a)}\equiv (0,a)\times\omega.
\]

The diffeomorphism $F_\eps$ defined in Subsection \ref{ss-cyl} induces the unitary transformation
\begin{gather*}
	\Phi_\eps:L^2(V_{\eps,(0,a)})\to L^2\big(\Pi_a,\eps^{d-1}s^{(d-1)q}\dd s \,\dd\H^{d-1}(t)\big),\\
	\qquad  \big(\Phi_\eps u\big)(s,t):=u\big(F_\eps(s,t)\big). 
\end{gather*}
Consider the symmetric bilinear form $p_{\eps,a}$ in $L^2\big(\Pi_a,\eps^{d-1}s^{(d-1)q}\dd s \,\dd\H^{d-1}(t)\big)$ defined by
\[
p_{\eps,a}(u,u):=t_{\eps,a}(\Phi_\eps^{-1} u,\Phi_\eps^{-1} u),
\]
on the domain $\cD(p_{\eps,a})=\Phi_\eps\big(\cD(t_{\eps,a})\big)$
and the associated self-adjoint operator $P_{\eps,a}$ in the weighted space $L^2\big(\Pi_a,\eps^{d-1}s^{(d-1)q}\dd s \,\dd\H^{d-1}(t)\big)$, which is by construction unitary
equivalent to $T_{\eps,a}$ and, hence, has the same eigenvalues.
Remark that due to the explicit form of $F_\eps$ and $\Phi_\eps$ we have
$\Phi_\eps\big(W^{1,2}_I(V_{\eps,I})\big)=W^{1,2}_I(\Pi_I)$.
It will be convenient to denote
\[
\cD_a:=W^{1,2}_{(0,a)}(\Pi_a),
\]
then \eqref{eq-core} shows that
\begin{equation*}
	\cD_a\text{ is a core domain of }p_{\eps,a},
	\label{eq-core2}
\end{equation*}
so $\cD_a$ can be used as a test domain when applying the min-max principle to~$P_{\eps,a}$.

Due to Lemma~\ref{two side est of bdy integral} and Lemma~\ref{two side est of gradient} for any $u\in\cD_a$ one has
\[
p^-_{\eps,a}(u,u)\le p_{\eps,a}(u,u)\le p^+_{\eps,a}(u,u)
\]
with symmetric bilinear forms $p^\pm_{\eps,a}(u,u)$ in $L^2\big(\Pi_a,\eps^{d-1}s^{(d-1)q}\dd s \,\dd\H^{d-1}(t)\big)$ defined on $\cD_a$
\begin{align*}
	p^-_{\eps,a}(u,u)&:=\eps^{d-1}\int_0^a s^{q(d-1)}\int_{\omega}\Bigg[\big(1-(d-1)\eps qR_\omega\big)|\partial_su|^2\\
	&\quad +\left(\dfrac{1}{\eps^2s^{2q}}-\frac{qR_\omega}{s^2\eps}-(d-1)\frac{q^2R^2_\omega}{s^2}\right)|\nabla_tu|^2\Bigg]\,\dd\H^{d-1}(t)\,\dd s\\
	&\quad- \eps^{d-2}\int_0^a
	s^{q(d-2)}\sqrt{1+\eps^2q^2\,R_\omega^2\,s^{2q-2}}\int_{\partial\omega}\big|u(s,t)\big|^2\,\dd\H^{d-2}(t)\,\dd s,\\
	p^+_{\eps,a}(u,u)&:=\eps^{d-1}\int_0^a s^{q(d-1)}\int_{\omega}\Bigg[\big(1+(d-1)\eps qR_\omega\big)|\partial_su|^2\\
	&\quad +\left(\dfrac{1}{\eps^2s^{2q}}+\frac{qR_\omega}{s^2\eps}+(d-1)\frac{q^2R^2_\omega}{s^2}\right)|\nabla_tu|^2\Bigg]\,\dd\H^{d-1}(t)\,\dd s\\
	&\quad -\eps^{d-2}\int_0^a
	s^{q(d-2)}\int_{\partial\omega}\big|u(s,t)\big|^2\,\dd\H^{d-2}(t)\,\dd s,
\end{align*}

In order to deal with $L^2$-spaces without weights we additionally consider the unitary
transform
\begin{gather*}
\cV_\eps:\ L^2(\Pi_a)\to L^2(\Pi_a,\eps^{d-1}s^{q(d-1)}\dd s\,\dd\H^{d-1}(t)\big),\\
\cV_\eps u(s,t):=\eps^{-\frac{d-1}{2}}s^{-\frac{q(d-1)}{2}}u(s,t),
\end{gather*}
and the symmetric bilinear forms $\Tilde p^{\,\pm}_{\eps,a}$ in $L^2(\Pi_a)$ defined
on $\cV_\eps^{-1} \cD_a \equiv \cD_a$
by
\[
\Tilde p^{\,\pm}_{\eps,a}(u,u):=p^\pm_{\eps,a}(\cV_\eps u,\cV_\eps u),
\]
i.e. for $u\in  \cD_a$ one has
\begin{align*}
\Tilde p^{\,-}_{\eps,a}(u,u)&=\int_0^a \int_{\omega}\Bigg[\big(1-(d-1)\eps qR_\omega\big)\Big(\partial_su -\frac{q(d-1)}{2s}\, u\Big)^2\\
&\qquad\qquad +\left(\dfrac{1}{\eps^2s^{2q}}-\frac{qR_\omega}{s^2\eps}-(d-1)\frac{q^2R^2_\omega}{s^2}\right)|\nabla_tu|^2\Bigg]\,\dd\H^{d-1}(t)\,\dd s\\
&\quad- \int_0^a \dfrac{1}{\eps s^q}
\sqrt{1+\eps^2q^2\,R_\omega^2\,s^{2q-2}}\int_{\partial\omega}\big|u(s,t)\big|\,\dd\H^{d-2}(t)\,\dd s,\\
\Tilde p^{\,+}_{\eps,a}(u,u)&=\int_0^a\int_{\omega}\Bigg[\big(1+(d-1)\eps qR_\omega\big)\Big(\partial_su -\frac{q(d-1)}{2s}\, u\Big)^2
	\\ &\qquad\qquad +\left( \dfrac{1}{\eps^2s^{2q}}+\frac{qR_\omega}{\eps s^2}+(d-1)\frac{q^2R^2_\omega}{s^2}\right)|\nabla_tu|^2\Bigg]\dd\H^{d-1}(t)\dd s\\
	&\quad -\int_0^a\dfrac{1}{\eps s^q} \int_{\partial\omega} u^2\,\dd\H^{d-2}(t)\,\dd s.
\end{align*}
Assuming that
\[
\eps\in(0,\eps_0)
\]
with some sufficiently small $\eps_0>0$ we can find suitable constants $c_j>0$,  $j\in\{1,2\}$, with $c_j\eps_0<1$
such that for all $s\in(0,a)$ one has
\[
\frac{qR_\omega}{\eps s^2}+(d-1)\frac{q^2R^2_\omega}{s^2}\le c_1\eps\cdot \dfrac{1}{\eps^2 s^{2q}},
\]
which yields,
\begin{align*}
	\dfrac{1}{\eps^2s^{2q}}+\frac{qR_\omega}{\eps s^2}+(d-1)\frac{q^2R^2_\omega}{s^2}&\le \dfrac{1+c_1\eps}{\eps^2 s^{2q}},\\
	\dfrac{1}{\eps^2s^{2q}}-\frac{qR_\omega}{\eps s^2}-(d-1)\frac{q^2R^2_\omega}{s^2}&\ge \dfrac{1-c_1\eps}{\eps^2 s^{2q}},
\end{align*}
and
\[
\sqrt{1+\eps^2q^2\,R_\omega^2\,s^{2q-2}}\le \dfrac{1}{1-c_2\eps}.
\]
This gives, with $c_0:=(d-1)qR_\omega$,
\begin{align*}
\Tilde p^{\,-}_{\eps,a}(u,u)&\ge (1-c_0\eps)\int_0^a \int_{\omega}\Big(\partial_su -\frac{q(d-1)}{2s}\, u\Big)^2\dd\H^{d-1}(t)\,\dd s\\
	&\quad +(1-c_1\eps)\int_0^a \dfrac{1}{\eps^2s^{2q}}\int_{\omega}|\nabla_tu|^2\,\dd\H^{d-1}(t)\,\dd s\\
	&\quad- \dfrac{1}{1-c_2\eps}\int_0^a \dfrac{1}{\eps s^q}
	\int_{\partial\omega}u^2\,\dd\H^{d-2}(t)\,\dd s,\\
\Tilde p^{\,+}_{\eps,a}(u,u)&\le (1+c_0\eps)\int_0^a \int_{\omega}\Big(\partial_su -\frac{q(d-1)}{2s}\, u\Big)^2\dd\H^{d-1}(t)\,\dd s\\
	&\quad +(1+c_1\eps)\int_0^a \dfrac{1}{\eps^2s^{2q}}\int_{\omega}|\nabla_tu|^2\,\dd\H^{d-1}(t)\,\dd s-\int_0^a \dfrac{1}{\eps s^q}
	\int_{\partial\omega}u^2\,\dd\H^{d-2}(t)\,\dd s.
\end{align*}

Due to $u\in \cD_a$ the integration by parts in $s$ gives
\[
\int_{0}^a\frac{u\,\partial_su}{s}\dd s=\int_0^a\frac{u^2}{2s^2}\dd s,
\]
resulting in
\begin{align}
\int_0^a \int_{\omega}\Big(\partial_su &-\frac{q(d-1)}{2s}\, u\Big)^2\dd\H^{d-1}(t)\,\dd s=\int_0^a \int_{\omega}\bigg(|\partial_su|^2+\dfrac{H}{s^2}\,u^2\bigg)\dd\H^{d-1}(t)\,\dd s,\nonumber\\
H:&=\dfrac{q^2(d-1)^2-2q(d-1)}{4}\equiv \dfrac{\big(q(d-1)-1\big)^2-1}{4}.\label{hhh}
\end{align}

By taking
\[
c:=\max\{c_0,c_1,c_2\}
\]
and by adjusting the value of $\eps_0$ to have $c\eps_0<1$ we arrive at
the inequalities
\[
h^-_{\eps,a}(u,u)\le \Tilde p^{\,-}_{\eps,a}(u,u), \qquad
\Tilde p^{\,+}_{\eps,a}(u,u)\le h^{\,+}_{\eps,a}(u,u)
\]
valid for all $\eps\in(0,\eps_0)$ and all $u\in\cD_a$, where
the symmetric bilinear forms $h^\pm_{\eps,a}$ in $L^2(\Pi_a)$ are defined on $\cD_a$ by
\begin{align*}
h^{-}_{\eps,a}(u,u)&= (1-c\eps)\int_0^a \int_{\omega}\bigg(|\partial_su|^2+\dfrac{H}{s^2}\,u^2\bigg)\dd\H^{d-1}(t)\,\dd s\\
&\quad +(1-c\eps)\int_0^a \dfrac{1}{\eps^2s^{2q}}\int_{\omega}|\nabla_tu|^2\,\dd\H^{d-1}(t)\,\dd s\\
&\quad- \dfrac{1}{1-c\eps}\int_0^a \dfrac{1}{\eps s^q}
\int_{\partial\omega}u^2\,\dd\H^{d-2}(t)\,\dd s,\\
h^{+}_{\eps,a}(u,u)&= (1+c\eps)\int_0^a \int_{\omega}\bigg(|\partial_su|^2+\dfrac{H}{s^2}\,u^2\bigg)\dd\H^{d-1}(t)\,\dd s\\
&\quad +(1+c\eps)\int_0^a \dfrac{1}{\eps^2s^{2q}}\int_{\omega}|\nabla_tu|^2\,\dd\H^{d-1}(t)\,\dd s-\int_0^a \dfrac{1}{\eps s^q}
\int_{\partial\omega}u^2\,\dd\H^{d-2}(t)\,\dd s.
\end{align*}

Using the min-max principle one summarizes the above considerations as follows:
\begin{lemma}\label{lem35}
	For any $a>0$ there exist $\eps_0>0$ and $c>0$, with $c \eps_0<1$, such that for any $j\in\N$ and any $\eps\in(0,\eps_0)$ it holds that
	\[
	\mu^-_j(\eps,a)\le \lambda_j(T_{\eps,a})\le \mu^+_j(\eps,a)
	\quad
	\text{with}
	\quad
	\mu^\pm_j(\eps,a):=\inf_{\substack{S\subset\cD_a\\ \dim S=j}}\sup_{u\in S\setminus\{0\}}\dfrac{h^\pm_{\eps,a}(u,u)}{\|u\|^2_{L^2(\Pi_a)}}.
	\]
\end{lemma}

\subsection{Upper bound}
In this subsection we are going to obtain an upper bound for the quantities $\mu^+_j(\eps,a)$
defined in Lemma~\ref{lem35}. The analysis is based on the observation that for any
$u\in \cD_a$ one has
\begin{align*}
	&h^{+}_{\eps,a}(u,u)= (1+c\eps)\int_0^a \int_{\omega}\bigg(|\partial_su|^2+\dfrac{H}{s^2}\,u^2\bigg)\dd\H^{d-1}(t)\,\dd s\\
	&\quad +(1+c\eps)\int_0^a \dfrac{1}{\eps^2s^{2q}}\bigg[\int_{\omega}|\nabla_tu|^2\,\dd\H^{d-1}(t) -\dfrac{\eps s^q}{1+c\eps}
	\int_{\partial\omega}u^2\,\dd\H^{d-2}(t)\bigg]\,\dd s.
\end{align*}
In other words, if we denote
\[
\rho_\eps(s):=\dfrac{s^q}{1+c\eps},
\]
then
\begin{equation}
	\label{hplus1}	
\begin{aligned}
h^{+}_{\eps,a}(u,u)&= (1+c\eps)\int_0^a \int_{\omega}\bigg(|\partial_su|^2+\dfrac{H}{s^2}\,u^2\bigg)\dd\H^{d-1}(t)\,\dd s\\
&\quad
+(1+c\eps)\int_0^a \dfrac{1}{\eps^2s^{2q}} r^\omega_{\eps\rho_\eps(s)}\big(u(s,\cdot),u(s,\cdot)\big)\dd s.
\end{aligned}
\end{equation}

Let $\psi_\alpha$ be the positive and $L^2$-normalized eigenfunction of the Robin Laplacian $R^\omega_\alpha$ for the first eigenvalue $\lambda_1(R^\omega_\alpha)$. In view of the above representation one is interested in $\psi_{\eps \rho_\eps(s)}$. By the assertion (iii) of Lemma \ref{basic properties}, for any $\eps>0$ the map
\[
(0,a)\ni s\mapsto \psi_{\eps\rho_\eps(s)}\in L^2(\omega)
\]
is smooth. For any function $f\in W^{1,2}_{(0,a)}(0,a)$ consider the associated function
\begin{equation}
	\label{eq-fu}
u:\ \Pi_a\ni(s,t)\mapsto f(s)\psi_{\eps \rho_\eps(s)}(t).
\end{equation}
Obviously $u$ belongs to $L^2(\Pi_a)$, its weak derivatives in $\Pi_a$ are given by
\begin{align*}
\partial_s u:&\ (s,t)\mapsto f'(s)\psi_{\eps \rho_\eps(s)}(t)+f(s)\dfrac{\dd \psi_{\eps\rho_\eps(s)}}{\dd s}(t),\\
\nabla_t u:&\ (s,t)\mapsto f(s)\nabla_t \psi_{\eps \rho_\eps(s)}(t),
\end{align*}
and, therefore, also belongs to $L^2(\Pi_a)$, which shows that $u\in \cD_a$.

\begin{lemma}\label{lem41}
	For any $a>0$ and $\eps_0>0$ there exists $B>0$ such that for any
	$f\in W^{1,2}_{(0,a)}(0,a)$, any $\eps\in(0,\eps_0)$ and for $u$ given by \eqref{eq-fu} it holds
\begin{align*}
	\int_{\Pi_a}u^2\dd\H^d&=\int_0^a f(s)^2\dd s,\\	
	h^+_{\eps,a}(u,u)&\le (1+c\eps)\int_0^a \bigg[f'(s)^2 +\bigg(\dfrac{H}{s^2}-\dfrac{A_\omega}{\eps(1+c\eps)s^q}+B\bigg)f(s)^2\bigg]\dd s. \label{hplus3}
\end{align*}
\end{lemma}

\begin{proof}
The normalization of $\psi_{\eps\rho_\eps(s)}$ gives
\begin{align*}
	\int_\omega u(s,t)^2\dd \H^{d-1}(t)&=f(s)^2\int_\omega \psi_{\eps\rho_\eps(s)}(t)^2\dd \H^{d-1}(t)=f(s)^2,\\
	\int_{\Pi_a}u^2\dd\H^d&=\int_0^a 	\int_\omega u(s,t)^2\dd \H^{d-1}(t)\dd s=\int_0^a f(s)^2\dd s.
\end{align*}
Due to the choice of $\psi_{\eps\rho_\eps(s)}$ we have
\[
r^\omega_{\eps\rho_\eps(s)}\big(u(s,\cdot),u(s,\cdot)\big)=\lambda_1\big(R^\omega_{\eps\rho_\eps(s)}\big)
\int_\omega u(s,t)^2\dd \H^{d-1}(t)= \lambda_1(R^\omega_{\eps\rho_\eps(s)})f(s)^2.
\]
Finally,
\begin{align*}
\int_\omega |\partial_s u|^2\dd\H^{d-1}(t)&=
\int_\omega \Big( f'(s)\psi_{\eps \rho_\eps(s)}(t)+f(s)\dfrac{\dd \psi_{\eps\rho_\eps(s)}}{\dd s}(t)\Big)^2\dd\H^{d-1}(t)\\
&=f'(s)^2+2f(s)f'(s)\int_\omega \psi_{\eps \rho_\eps(s)}(t)\dfrac{\dd \psi_{\eps\rho_\eps(s)}}{\dd s}(t)\dd\H^{d-1}(t)\\
&\qquad+f(s)^2\int_\omega \Big|\dfrac{\dd \psi_{\eps\rho_\eps(s)}}{\dd s}(t)\Big|^2\dd\H^{d-1}(t),
\end{align*}
and using 
\[
\int_\omega \psi_{\eps \rho_\eps(s)}(t)\dfrac{\dd \psi_{\eps\rho_\eps(s)}}{\dd s}(t)\,\dd\H^{d-1}(t)
=\dfrac{1}{2}\dfrac{\dd}{\dd s}\int_\omega \psi_{\eps \rho_\eps(s)}(t)^2\dd\H^{d-1}(t)
=\dfrac{1}{2}\dfrac{\dd}{\dd s} 1=0
\]
to eliminate the middle term on the right-hand side
we obtain
\begin{gather*}
\int_\omega |\partial_s u|^2\dd\H^{d-1}(t)=f'(s)^2+W_\eps f(s)^2,\qquad
W_\eps(s):=	\int_\omega \Big|\dfrac{\dd \psi_{\eps\rho_\eps(s)}}{\dd s}(t)\Big|^2\dd\H^{d-1}(t).
\end{gather*}

The substitution of the previous identities into \eqref{hplus1} yields
\begin{equation}
	\label{hplus2}
h^+_{\eps,a}(u,u)=(1+c\eps)\int_0^a \left[f'(s)^2 +\left(W_\eps(s)+\dfrac{H}{s^2} +\dfrac{\lambda_1\left(R^\omega_{\eps\rho_\eps(s)}\right)}{\eps^2 s^{2q}}\right) f(s)^2\right]\dd s.
\end{equation}

Using Lemma \ref{basic properties}(iv) we estimate, with a suitable $C_1>0$,
\begin{align*}
\dfrac{\lambda_1(R^\omega_{\eps\rho_\eps(s)})}{\eps^2 s^{2q}}&
\le \dfrac{-A_\omega \eps\rho_\eps(s) +C_1\eps^2\rho_\eps(s)^2}{\eps^2s^{2q}}\\
&=-A_\omega \dfrac{\rho_\eps(s)}{s^{2q}}\cdot\dfrac{1}{\eps}+C_1\dfrac{\rho_\eps(s)^2}{s^{2q}}\le -\dfrac{A_\omega}{\eps(1+c\eps)}\cdot\dfrac{1}{s^q}+C_1.
\end{align*}
Furthermore,
\begin{equation*}
	\begin{split}
		W_\eps(s)
		&=\int_{\omega}\bigg(\dfrac{\dd \psi_\sigma}{\dd\sigma}(t)\Big|_{\sigma=\eps\rho_\eps(s)}
		\dfrac{\dd\big(\eps \rho_\eps(s)\big)}{\dd s}\bigg)^2\dd\H^{d-1}(t)=\frac{q^2s^{2q-2}}{(1+c\eps)^2}\,\eps^2\bigg\|\dfrac{\dd \psi_\sigma}{\dd\sigma}\Big|_{\sigma=\eps\rho_\eps(s)}\bigg\|^2_{L^2(\omega)}.
	\end{split}
\end{equation*}
By Lemma \ref{basic properties}(iii) we can find $C_2>0$ such that
\[
\bigg\|\dfrac{\dd \psi_\sigma}{\dd\sigma}\Big|_{\sigma=\eps\rho_\eps(s)}\bigg\|^2_{L^2(\omega)}\le C_2 \quad\text{for all $s\in(0,a)$ and $\eps\in(0,\eps_0)$,}
\]
therefore,
\[
W_\eps(s)\le \dfrac{C_2q^2s^{2q-2}}{(1+c\eps)^2}\eps^2\le C_2 q^2 a^{2q-2}\eps_0^2=:C_3.
\]
Making use of these inequalities in \eqref{hplus2} one arrives at the claim with $B:=C_1+C_3$.
\end{proof}

\begin{lemma}\label{lem10}
For any $a>0$ and $j\in\N$ one has
\[
\lambda_j(T_{\eps,a})\le \left(\dfrac{A_\omega}{\eps}\right)^{\frac{2}{2-q}}\lambda_j(L_1)+O\left(\dfrac{1}{\eps}\right)^{\frac{q}{2-q}} \text{ for }\eps\to 0^+.
\]		
\end{lemma}

\begin{proof}
Let $j\in \N$ and $S\subset W^{1,2}_{(0,a)}(0,a)$ be a $j$-dimensional subspace. Due to the first identity in Lemma~\ref{lem41} the set
\[
\Tilde S:=\Big\{ u:\  u(s,t)=f(s)\psi_{\eps \rho_\eps(s)}(t) \text{ with }f\in S\Big\}
\]
is a $j$-dimensional subspace of $\cD_a$. Therefore,
\[
\mu_j^+(\eps,a)\le \sup_{u\in \Tilde S\setminus\{0\}}\dfrac{h^+_{\eps,a}(u,u)}{\|u\|^2_{L^2(\Pi_a)}},
\]
and using Lemma \ref{lem41} one obtains
\begin{align*}
\dfrac{\mu_j^+(\eps,a)}{1+c\eps}&\le	\sup_{u\in S\setminus\{0\}}\dfrac{\displaystyle\int_0^a \bigg[f'(s)^2 +\bigg(\dfrac{H}{s^2}-\dfrac{A_\omega}{\eps(1+c\eps)s^q}\bigg)f(s)^2\bigg]\dd s}{\|f\|^2_{L^2(0,a)}}+B.
\end{align*}
The constants $c$ and $B$ are independent of $S$, so one can take the infimum over all $S$ as above 
to arrive at
\begin{align*}
	\dfrac{\mu_j^+(\eps,a)}{1+c\eps}&\le	\inf_{\substack{S\subset W^{1,2}_{(0,a)}(0,a)\\ \dim S=j}}\sup_{u\in S\setminus\{0\}}\dfrac{\displaystyle\int_0^a \bigg[f'(s)^2 +\bigg(\dfrac{H}{s^2}-\dfrac{A_\omega}{\eps(1+c\eps)s^q}\bigg)f(s)^2\bigg]\dd s}{\|f\|^2_{L^2(0,a)}} + B.
\end{align*}
The first summand on the right-hand side is exactly the characterization of the eigenvalue $\lambda_j\left(L_{\frac{A_\omega}{c(1+c\eps)},a}\right)$ using the min-max principle, which yields
\[
\mu_j^+(\eps,a)\le (1+c\eps)\lambda_j\left(L_{\frac{A_\omega}{\eps(1+c\eps)},a}\right)+(1+c\eps)B.
\]
With the help of the upper bound for $\lambda_j(T_{\eps,a})$ and Lemma \ref{modelop1} with $\eps\to 0^+$ we obtain, with some $K>0$,
\begin{align*}
\lambda_j(T_{\eps,a})&\le \mu_j^+(\eps,a)
\le (1+c\eps)\left[\left(\dfrac{A_\omega}{\eps(1+c\eps)}\right)^{\frac{2}{2-q}} \lambda_j(L_1)+K\right]+(1+c\eps)B,
\end{align*}
which gives the sought estimate.	
\end{proof}

\subsection{Lower bound}

Similarly to the upper bound, the subsequent estimates for the numbers $\mu^-_j(\eps,a)$
defined in Lemma~\ref{lem35} will be based on the observation that for any $u\in\cD_a$
one has
\begin{equation}
	\label{hplus1b}	
	\begin{aligned}
		h^{-}_{\eps,a}(u,u)&= (1-c\eps)\int_0^a \int_{\omega}\bigg(|\partial_su|^2+\dfrac{H}{s^2}\,u^2\bigg)\dd\H^{d-1}(t)\,\dd s\\
		&\quad
		+(1-c\eps)\int_0^a \dfrac{1}{\eps^2s^{2q}} r^\omega_{\eps\rho_\eps(s)}\big(u(s,\cdot),u(s,\cdot)\big)\dd s,\\
\text{with }	\rho_\eps(s)&:=\dfrac{s^q}{(1-c\eps)^2}.
	\end{aligned}
\end{equation}
As above, let $\psi_\alpha$ be the positive and $L^2$-normalized eigenfunction of the Robin Laplacian $R^\omega_\alpha$ for the first eigenvalue $\lambda_1(R^\omega_\alpha)$.
We represent each function $u\in \cD_a$ as
\begin{align}
	u&=v+w, \label{eq-uvw}\\
	v(s,t)&:=f(s)\psi_{\eps\rho_\eps(s)}(t),\nonumber\\
	f(s)&:=\int_\omega \psi_{\eps\rho_\eps(s)}(t)u(s,t)\,\dd\H^{d-1}(t),\nonumber\\
w&:=u-v. \nonumber
\end{align}
By construction one has $f\in W^{1,2}_{(0,a)}(0,a)$,
and due to Lemma \ref{basic properties}(iii) the weak derivatives of $v$ in $\Pi_a$ are given by
\begin{align}
	\partial_s v:&\ (s,t)\mapsto f'(s)\psi_{\eps \rho_\eps(s)}(t)+f(s)\dfrac{\dd \psi_{\eps\rho_\eps(s)}}{\dd s}(t), \label{psv}\\
	\nabla_t v:&\ (s,t)\mapsto f(s)\nabla_t \psi_{\eps \rho_\eps(s)}(t).\nonumber
\end{align}
In particular, the functions $u$, $\partial_s u$ and $\partial_{t_j} u$ belong to $L^2(\Pi_a)$,
so one has $v,w\in \cD_a$.

\begin{lemma}\label{lem43}
There exist a sufficiently small $\eps_0>0$ and a sufficiently large $B>0$ such that
for any $\eps\in(0,\eps_0)$ and any $u\in \cD_a$ decomposed as in \eqref{eq-uvw} it holds that
\begin{align*}
\|u\|^2_{L^2(\Pi_a)}&=\|f\|^2_{L^2(0,a)}+\|w\|^2_{L^2(\Pi_a)},\\
\dfrac{h^-_{\eps,a}(u,u)}{1-c\eps}&\ge (1-B\eps)\int_0^a\bigg[f'(s)^2  +\bigg(\dfrac{H}{s^2}
-\dfrac{A_\omega}{\eps (1-B\eps)^2 s^q}\bigg)	f(s)^2	\bigg]\dd s
-B\|f\|^2_{L^2(0,a)}.
\end{align*}
\end{lemma}

\begin{proof}
Let us collect important properties of the decomposition \eqref{eq-uvw}.
First,
\begin{align*}
	\int_\omega v(s,t)^2\dd \H^{d-1}(t)&=f(s)^2\int_\omega \psi_{\eps\rho_\eps(s)}(t)^2\dd \H^{d-1}(t)=f(s)^2,\\
	\int_{\Pi_a}v^2\dd\H^d&=\int_0^a 	\int_\omega v(s,t)^2\dd \H^{d-1}(t)\dd s=\int_0^a f(s)^2\dd s.
\end{align*}
Furthermore, by construction we have
\begin{equation}
	\label{vw00}
\int_\omega \psi_{\eps\rho_\eps(s)}(t)w(s,t)\dd \H^{d-1}(t)=0,
\end{equation}
in particular,
\begin{equation}
	\label{eq-vwperp}
\int_\omega v(s,t)w(s,t)\,\dd \H^{d-1}(t)=f(s)\int_\omega \psi_{\eps\rho_\eps(s)}(t)w(s,t)\,\dd \H^{d-1}(t)
=0,
\end{equation}
therefore,
\begin{align*}
	\int_\omega u(s,t)^2\dd\H^{d-1}(t)&=\int_\omega v(s,t)^2\dd \H^{d-1}(t) +2 \int_\omega v(s,t)w(s,t)\dd \H^{d-1}(t)\\
	&\quad+\int_\omega w(s,t)^2\dd \H^{d-1}(t)\\
	&=\int_\omega v(s,t)^2\dd \H^{d-1}(t) +\int_\omega w(s,t)^2\dd \H^{d-1}(t)\\
	&=f(s)^2+\int_\omega w(s,t)^2\dd \H^{d-1}(t),\\
	\|u\|^2_{L^2(\Pi_a)}&=\|f\|^2_{L^2(0,a)}+\|w\|^2_{L^2(\Pi_a)}.
\end{align*}
Due to the choice of $\psi_{\eps\rho_\eps(s)}$ and the orthogonality  \eqref{eq-vwperp}, the spectral theorem for $R^\omega_{\eps,\rho_\eps(s)}$ gives
\begin{align*}
r^\omega_{\eps\rho_\eps(s)}\big(u(s,\cdot),u(s,\cdot)\big)&=
r^\omega_{\eps\rho_\eps(s)}\big(v(s,\cdot),v(s,\cdot)\big)+r^\omega_{\eps\rho_\eps(s)}\big(w(s,\cdot),w(s,\cdot)\big)\\
&\ge
\lambda_1(R^\omega_{\eps\rho_\eps(s)}) \big\| v(s,\cdot)\big\|^2_{L^2(\omega)}
+\lambda_2(R^\omega_{\eps\rho_\eps(s)}) \big\| w(s,\cdot)\big\|^2_{L^2(\omega)}\\
&=\lambda_1(R^\omega_{\eps\rho_\eps(s)}) f(s)^2
+\lambda_2(R^\omega_{\eps\rho_\eps(s)}) \big\| w(s,\cdot)\big\|^2_{L^2(\omega)}.
\end{align*}
Using these estimates in \eqref{hplus1b} one obtains
\begin{align*}
	\dfrac{h^-_{\eps,a}(u,u)}{1-c\eps}&\ge \int_0^a\int_\omega |\partial_s u|^2\dd\H^{d-1}(t)\dd s+\int_0^a\bigg(\dfrac{H}{s^2}+\dfrac{\lambda_1(R^\omega_{\eps\rho_\eps(s)})}{\eps^2 s^{2q}}\bigg)f(s)^2\dd s\\
	&\quad+\int_0^a\int_\omega\bigg(\dfrac{H}{s^2}+\dfrac{\lambda_2(R^\omega_{\eps\rho_\eps(s)})}{\eps^2 s^{2q}}\bigg) w(s,t)^2\dd\H^{d-1}(t)\dd s.
\end{align*}
We first use Lemma \ref{basic properties}(iv) to estimate with a suitable $C_1>0$:
\begin{align*}
	\dfrac{\lambda_1(R^\omega_{\eps\rho_\eps(s)})}{\eps^2 s^{2q}}&
	\ge -\dfrac{A_\omega \eps\rho_\eps(s) +C_1\eps^2\rho_\eps(s)^2}{\eps^2s^{2q}}\\
	&=-A_\omega \dfrac{\rho_\eps(s)}{s^{2q}}\cdot\dfrac{1}{\eps}-C_1\dfrac{\rho_\eps(s)^2}{s^{2q}}\ge -\dfrac{A_\omega}{\eps(1-c\eps)^2}\cdot\dfrac{1}{s^q}-\dfrac{C_1}{(1-c\eps)^4}.
\end{align*}
In addition, using Lemma \ref{basic properties}(v) we find some $C_2>0$ and, if necessary, decrease the value of $\eps_0$ to obtain
\[
\lambda_2\big(R^\omega_{\eps\rho_\eps(s)}\big)\ge C_2\text{ for all $\eps\in(0,\eps_0)$ and $s\in(0,a)$.}
\]
Then for a suitable $C_3>0$ and an adjusted value of $\eps_0$ one has
\[
\dfrac{H}{s^2}+\dfrac{\lambda_2(R^\omega_{\eps\rho_\eps(s)})}{\eps^2 s^{2q}}\ge \dfrac{C_3}{\eps^2 s^{2q}}\text{ for all $\eps\in(0,\eps_0)$ and $s\in(0,a)$,}
\]
yielding
\begin{equation}
	\label{hminus1}
\begin{aligned}
	\dfrac{h^-_{\eps,a}(u,u)}{1-c\eps}&\ge \int_0^a\int_\omega |\partial_s u|^2\dd\H^{d-1}(t)\,\dd s\\
	&\quad+\int_0^a\bigg(\dfrac{H}{s^2}
	-\dfrac{A_\omega}{\eps(1-c\eps)^2}\cdot\dfrac{1}{s^q}-\dfrac{C_1}{(1-c\eps)^4}
	\bigg)f(s)^2\dd s\\
	&\quad+\int_0^a\int_\omega \dfrac{C_3}{\eps^2 s^{2q}}
	 w(s,t)^2\dd\H^{d-1}(t)\,\dd s.
\end{aligned}
\end{equation}

Our next aim is to study the term containing $|\partial_s u|^2$ on the right-hand side of \eqref{hminus1}. By denoting
\begin{align*}
	I_1&:=\int_\omega |\partial_s v|^2\dd\H^{d-1}(t),\\
	I_2&:=2\int_\omega \partial_s v\cdot\partial_s w \,\dd\H^{d-1}(t),\\
	I_3&:=\int_\omega |\partial_s w|^2\dd\H^{d-1}(t),
\end{align*}
we arrive at the decomposition
\begin{equation}
	\label{udecomp}
	\int_\omega |\partial_s u|^2\dd\H^{d-1}(t)=I_1+I_2+I_3.
\end{equation}

The term $I_1$ is analyzed in a straightforward way:
\begin{align*}
	I_1&=
	\int_\omega \Big( f'(s)\psi_{\eps \rho_\eps(s)}(t)+f(s)\dfrac{\dd \psi_{\eps\rho_\eps(s)}}{\dd s}(t)\Big)^2\dd\H^{d-1}(t)\\
	&=f'(s)^2+2f(s)f'(s)\int_\omega \psi_{\eps \rho_\eps(s)}(t)\dfrac{\dd \psi_{\eps\rho_\eps(s)}}{\dd s}(t)\dd\H^{d-1}(t)\\
	&\qquad+f(s)^2\int_\omega \Big|\dfrac{\dd \psi_{\eps\rho_\eps(s)}}{\dd s}(t)\Big|^2\dd\H^{d-1}(t),
\end{align*}
and using 
\[
\int_\omega \psi_{\eps \rho_\eps(s)}(t)\dfrac{\dd \psi_{\eps\rho_\eps(s)}}{\dd s}(t)\,\dd\H^{d-1}(t)
=\dfrac{1}{2}\dfrac{\dd}{\dd s}\int_\omega \psi_{\eps \rho_\eps(s)}(t)^2\dd\H^{d-1}(t)
=\dfrac{1}{2}\dfrac{\dd}{\dd s} 1=0
\]
to eliminate the middle term on the right-hand side, and by noting that last term is nonnegative we obtain
\begin{equation}
	\label{eq-i1}
	I_1\ge f'(s)^2.
\end{equation}	

As a consequence of \eqref{psv}, we have
\begin{align*}
	I_2&=I'_2+I''_2,\\
	I'_2&=2f'(s)\int_\omega \psi_{\eps \rho_\eps(s)}(t)\partial_s w(s,t)\,\dd\H^{d-1}(t),\\
	I''_2&=2f(s)\int_\omega \dfrac{\dd\psi_{\eps \rho_\eps(s)}}{\dd s}(t)\partial_s w(s,t)\,\dd\H^{d-1}(t).
\end{align*}
Using the orthogonality \eqref{vw00} we obtain
\begin{align*}
0&=\dfrac{\dd }{\dd s} \int_\omega \psi_{\eps \rho_\eps(s)}(t) w(s,t)\,\dd\H^{d-1}(t)\\
&=\int_\omega \dfrac{\dd\psi_{\eps \rho_\eps(s)}}{\dd s}(t) w(s,t)\,\dd\H^{d-1}(t)
+\int_\omega \psi_{\eps \rho_\eps(s)}\partial_s w(s,t)\,\dd\H^{d-1}(t),
\end{align*}
therefore,
\[
\int_\omega \psi_{\eps \rho_\eps(s)}(t)\partial_s w(s,t)\,\dd\H^{d-1}(t)=-\int_\omega 
\dfrac{\dd\psi_{\eps \rho_\eps(s)}}{\dd s}(t)  w(s,t)\,\dd\H^{d-1}(t).
\]
This allows one to estimate $I'_2$ by
\begin{align*}
	|I'_2|&=\bigg|\int_\omega 2f'(s)\dfrac{\dd\psi_{\eps \rho_\eps(s)}}{\dd s}(t) w(s,t)\dd\H^{d-1}(t)\bigg|\\
	&\le
	W_\eps(s) f'(s)^2 +\int_\omega w(s,t)^2\dd\H^{d-1}(t)\\
	\text{with }W_\eps(s)&:=\int_\omega\bigg|\dfrac{\dd\psi_{\eps \rho_\eps(s)}}{\dd s}(t)\bigg|^2\dd\H^{d-1}(t).
\end{align*}
Furthermore,
\begin{align*}
	|I''_2|&\le \dfrac{W_\eps(s)}{\eps}\, f(s)^2+ \eps \int_\omega \partial_s w(s,t)^2\dd\H^{d-1}(t).
\end{align*}	
We have 
\begin{equation*}
	\begin{split}
		W_\eps(s)
		&=\int_{\omega}\bigg(\dfrac{\dd \psi_\sigma}{\dd\sigma}(t)\Big|_{\sigma=\eps\rho_\eps(s)}
		\dfrac{\dd\big(\eps \rho_\eps(s)\big)}{\dd s}\bigg)^2\dd\H^{d-1}(t)
		\\
		&=\frac{q^2s^{2q-2}}{(1-c\eps)^4}\,\eps^2\bigg\|\dfrac{\dd \psi_\sigma}{\dd\sigma}\Big|_{\sigma=\eps\rho_\eps(s)}\bigg\|^2_{L^2(\omega)},
	\end{split}
\end{equation*}
and by Lemma \ref{basic properties}(iii) we can find $C_4>0$ such that
\[
\bigg\|\dfrac{\dd \psi_\sigma}{\dd\sigma}\Big|_{\sigma=\eps\rho_\eps(s)}\bigg\|^2_{L^2(\omega)}\le C_4 \quad\text{for all $s\in(0,a)$ and $\eps\in(0,\eps_0)$,}
\]
then for some $C_5>0$ it holds
\[
W_\eps(s)\le C_5\eps^2\text{ for all $s\in(0,a)$ and $\eps\in(0,\eps_0)$.}
\]
Then
\begin{align*}
	|I'_2|&\le C_5\eps^2 f'(s)^2+\big\|w(s,\cdot)\big\|^2_{L^2(\omega)},\\
	|I''_2|&\le C_5\eps f(s)^2+\eps\big\|\partial_s w(s,\cdot)\big\|^2_{L^2(\omega)},
\end{align*}
and
\begin{align*}
	I_2&\ge -|I'_2|-|I''_2|\ge C_5\eps^2 f'(s)^2- C_5\eps f(s)^2-\eps\big\|\partial_s w(s,\cdot)\big\|^2_{L^2(\omega)}
	-\big\|w(s,\cdot)\big\|^2_{L^2(\omega)}.
\end{align*}
By using the last inequality and \eqref{eq-i1} in \eqref{udecomp} we arrive at
\begin{align*}
	\big\|\partial_s u(s,\cdot)\big\|^2_{L^2(\omega)}&\ge (1-C_5\eps^2)f'(s)^2-C_5\eps f(s)^2\\
	&\qquad +(1-\eps)\big\|\partial_s w(s,\cdot)\big\|^2_{L^2(\omega)}-\big\|w(s,\cdot)\big\|^2_{L^2(\omega)}.
\end{align*}
By using the last inequality to estimate the first summand on the right-hand side of \eqref{hminus1}
we obtain
\begin{align*}
		\dfrac{h^-_{\eps,a}(u,u)}{1-c\eps}&\ge 
		\int_0^a\bigg[(1-C_5\eps^2)f'(s)^2+\bigg(\dfrac{H}{s^2}
		-\dfrac{A_\omega}{\eps(1-c\eps)^2 s^q}-\dfrac{C_1}{(1-c\eps)^4}-C_5\eps\bigg)
		f(s)^2	\bigg]\dd s\\
		&\quad+\int_0^a\int_\omega \bigg[(1-\eps)\partial_s w(s,t)^2+\bigg(\dfrac{C_3}{\eps^2 s^{2q}}
		-1\bigg)w(s,t)^2\bigg]\dd\H^{d-1}(t)\,\dd s.
\end{align*}
We additionally adjust $\eps_0$ such that the expression in the last integral
becomes non-negative for all $s\in(0,a)$ and all $\eps\in(0,\eps_0)$ and choose $C_6>0$
to have
\[
\dfrac{C_1}{(1-c\eps)^4}+C_5\eps \le C_6 \text{ for all }\eps\in(0,\eps_0),
\]
which results in
\begin{equation}
	\label{hminus3}
	\dfrac{h^-_{\eps,a}(u,u)}{1-c\eps}\ge \int_0^a\bigg[(1-C_5\eps^2)f'(s)^2 
	+\bigg(\dfrac{H}{s^2}
	-\dfrac{A_\omega}{\eps(1-c\eps)^2 s^q}-C_6\bigg)
	f(s)^2	\bigg]\dd s.
\end{equation}
Recall that due to the explicit expression \eqref{hhh} for $H$ and the one-dimensional Hardy inequality we have
\begin{align*}
\dfrac{\big(q(d-1)-1\big)^2}{4}\int_0^a \dfrac{1}{s^2}\, f(s)^2\dd s
&\le \int_0^a \bigg(f'(s)^2+\dfrac{\big(q(d-1)-1\big)^2-1}{4 s^2}f(s)^2\bigg)\dd s\\
&\equiv \int_0^a \bigg(f'(s)^2+\dfrac{H}{s^2}f(s)^2\bigg)\dd s,
\end{align*}
yielding
\[
\int_0^a \dfrac{1}{s^2}\, f(s)^2\dd s\le C_7 \int_0^a \bigg(f'(s)^2+\dfrac{H}{s^2}f(s)^2\bigg)\dd s,
\quad C_7:=\dfrac{4}{\big(q(d-1)-1\big)^2}.
\]
This implies
\begin{align*}
\int_0^a&\bigg[(1-C_5\eps^2)f'(s)^2 +\dfrac{H}{s^2} f(s)^2\bigg]\dd s\\
&=(1-C_5\eps^2)\int_0^a\bigg(f'(s)^2+\dfrac{H}{s^2}f(s)^2\bigg)\dd s +H C_5 \eps^2 \int_0^a \dfrac{1}{s^2}\, f(s)^2\dd s\\
&\ge (1-C_5\eps^2)\int_0^a\bigg(f'(s)^2+\dfrac{H}{s^2}f(s)^2\bigg)\dd s- |H|C_5C_7\eps^2 \int_0^a \bigg(f'(s)^2+\dfrac{H}{s^2}f(s)^2\bigg)\dd s\\
&= (1-C_8\eps^2)\int_0^a\bigg(f'(s)^2+\dfrac{H}{s^2}f(s)^2\bigg)\dd s\quad
\text{ with }C_8:=C_5+|H|C_5C_7.
\end{align*}
By using this inequality on the right-hand side of \eqref{hminus3} we arrive at
\begin{align*}
	\dfrac{h^-_{\eps,a}(u,u)}{1-c\eps}&\ge (1-C_8\eps^2)\int_0^a\bigg(f'(s)^2 +\dfrac{H}{s^2}f(s)^2\bigg)\,\dd s-\int_0^a\dfrac{A_\omega}{\eps s^q (1-c\eps)^2}	f(s)^2\dd s
-C_6\|f\|^2_{L^2(0,a)}.
\end{align*}
We assume additionally $\eps_0\in(0,1)$, then $\eps^2<\eps$ for all $\eps\in(0,\eps_0)$, and we obtain the claim by choosing any $B\ge \max\{C_6,C_8\}$ such that
\[
\dfrac{1}{(1-c\eps)^2}\le\dfrac{1}{1-B\eps} \text{ for all }\eps\in(0,\eps_0) \qedhere
\]
holds.
\end{proof}

\begin{lemma}\label{lem12}
For any $a>0$ and any $j\in\N$ one has
\[
\lambda_j(T_{\eps,a})\geq \left(\dfrac{A_\omega}{\eps}\right)^{\frac{2}{2-q}}\lambda_j(L_1)+O\left(\dfrac{1}{\eps}\right)^{\frac{q}{2-q}} \text{ for }\eps\to 0^+.
\]
\end{lemma}

\begin{proof}
	Due to the first identity in Lemma \ref{lem43}, the map $u\mapsto (f,w)$ defined by \eqref{eq-uvw} is uniquely extended to an isometry
	\[
	J:\ L^2(\Pi_a)\to L^2(0,a)\oplus L^2(\Pi_a).
	\]
	Consider the symmetric bilinear form $m_\eps$ in  $L^2(0,a)\oplus L^2(\Pi_a)$ given by
\[
m_\eps\big((f,w),(f,w)\big):=(1-B\eps)\int_0^a\bigg[f'(s)^2  +\bigg(\dfrac{H}{s^2}
-\dfrac{A_\omega}{\eps (1-B\eps)^2 s^q}\bigg)	f(s)^2	\bigg]\dd s
-B\|f\|^2_{L^2(0,a)}
\]	
on $\cD(m_\eps):=W^{1,2}_{(0,a)}(0,a)\oplus L^2(\Pi_a)$,
then the inequality in Lemma \ref{lem43} can be read as
\[
\dfrac{h^-_{\eps,a}(u,u)}{1-c\eps}\ge m_\eps(Ju,Ju) \text{ for all }u\in \cD_a.
\]	
Further remark that the closure of $m_\eps$ is the bilinear form corresponding to the self-adjoint operator
\[
M_\eps:= \Big((1-B\eps)L_{\frac{A_\omega}{\eps (1-B\eps)^2},a}-B\Big)\oplus 0.
\]
Let $j\in\N$. For each $j$-dimensional subspace $S\subset \cD_a$ the set $J(S)$ is a $j$-dimensional subspace in the domain of $m_\eps$, and due to Lemma \ref{lem35} we have
\begin{equation}
	  \label{meps}
\begin{aligned}
\dfrac{\lambda_j(T_{\eps,a})}{1-c\eps}&\ge \dfrac{\mu^-_j(\eps,a)}{1-c\eps}
\equiv \inf_{\substack{S\subset\cD_a\\ \dim S=j}}\sup_{u\in S\setminus\{0\}}\dfrac{h^-_{\eps,a}(u,u)}{\|u\|^2_{L^2(\Pi_a)}}\\
&\ge \inf_{\substack{S\subset\cD_a\\ \dim S=j}}\sup_{u\in S\setminus\{0\}}\dfrac{m_\eps(Ju,Ju)}{\|Ju\|^2_{ L^2(0,a)\oplus L^2(\Pi_a)}}\\
&\equiv \inf_{\substack{S\subset\cD_a\\ \dim S=j}}\sup_{z\in J(S)\setminus\{0\}}\dfrac{m_\eps(z,z)}{\|z\|^2_{ L^2(0,a)\oplus L^2(\Pi_a)}}\\
&\ge \inf_{\substack{S'\subset\cD(m_\eps)\\ \dim S'=j}}\sup_{z\in S'\setminus\{0\}}\dfrac{m_\eps(z,z)}{\|z\|^2_{ L^2(0,a)\oplus L^2(\Pi_a)}}=\lambda_j(M_\eps).
\end{aligned}
\end{equation}
For $\eps\to 0^+$ we have, due to Lemma \ref{modelop1},
\begin{align*}
	\lambda_j\Big((1-B\eps)L_{\frac{A_\omega}{\eps (1-B\eps)^2},a}-B\Big)&
	=(1-B\eps)\lambda_j(L_{\frac{A_\omega}{\eps (1-B\eps)^2},a})-B\\
	&\ge(1-B\eps) \Big(\frac{A_\omega}{\eps (1-B\eps)^2}\Big)^\frac{2}{2-q}\lambda_j(L_1)-B\\
	&=\left(\dfrac{A_\omega}{\eps}\right)^{\frac{2}{2-q}}\lambda_j(L_1)+O\left(\dfrac{1}{\eps}\right)^{\frac{q}{2-q}},
\end{align*}
hence,
\begin{align*}
\lambda_j(M_\eps)&=\lambda_j\bigg(\Big[(1-B\eps)L_{\frac{A_\omega}{\eps (1-B\eps)^2},a}-B\Big]\oplus 0\bigg)
\ge \min\bigg\{ \lambda_j\Big((1-B\eps)L_{\frac{A_\omega}{\eps (1-B\eps)^2},a}-B\Big),0\bigg\}\\
&\ge \min\bigg\{ \left(\dfrac{A_\omega}{\eps}\right)^{\frac{2}{2-q}}\lambda_j(L_1)+O\left(\dfrac{1}{\eps}\right)^{\frac{q}{2-q}},0\bigg\}
=\left(\dfrac{A_\omega}{\eps}\right)^{\frac{2}{2-q}}\lambda_j(L_1)+O\left(\dfrac{1}{\eps}\right)^{\frac{q}{2-q}},
\end{align*}
and the substitution into \eqref{meps} completes the proof.
\end{proof}

By combining the upper bound of Lemma \ref{lem10} and the lower bound of Lemma \ref{lem12}
and by recalling the convention \eqref{tea} we arrive
at the main result of this section:

\begin{corollary}\label{corol14}
	For any $a>0$ and $j\in\N$ one has
	\[
	\lambda_j(T^{1,D}_{\eps,(0,a)})= \left(\dfrac{A_\omega}{\eps}\right)^{\frac{2}{2-q}}\lambda_j(L_1)+O\left(\dfrac{1}{\eps}\right)^{\frac{q}{2-q}} \text{ for }\eps\to 0^+.
	\]		
\end{corollary}

	\section{Truncations and the proof of Theorem \ref{thmain} (eigenvalue asymptotics)}\label{secproof}
	
The argument will be based on a series of truncations
combined with the spectral analysis of the operators $T^{\alpha,N/D}_{\eps,I}$ from the previous sections.
	
Choose $\delta>0$ for $\Omega$ as in Definition~\ref{defin1}, then pick any $a\in(0,\delta)$ and consider the sets
\[
\Omega_a:=\Omega\cap\Big[(-a,a)\times (-\delta,\delta)^{d-1}\Big],
\qquad
\Omega'_a:=\Omega\setminus\overline{\Omega_a}.
\]
By assumption on $\Omega$ one has $\Omega_a=V_{1,(0,a)}$,
while $\Omega'_a$ is a bounded Lipschitz domain.

\begin{lemma}\label{lem15}	For any $j\in\N$ one has
	\[
		\lambda_j(R^\Omega_\alpha)\le A_\omega^\frac{2}{2-q} \lambda_j(L_1) \alpha^\frac{2}{2-q}+O\Big(\alpha^{\frac{2}{2-q}-(q-1)}\Big) \text{ for $\alpha\to +\infty$.}
	\]	
\end{lemma}

\begin{proof}
As each function from $\widehat W^{1,2}_0(V_{(1,(0,a)})$ can be extended by zero to a function in $W^{1,2}(\Omega)$, the min-max principle and the scaling \eqref{lscal} imply that for any $j\in\N$ and any $\alpha>0$ it holds that
	\[
	\lambda_j(R^\Omega_\alpha)\le \lambda_j(T^{\alpha,D}_{1,(0,a)})=\alpha^2 \lambda_j(T^{1,D}_{\alpha^{1-q},(0,\alpha a)}).
	\]
Now assume that $\alpha>1$ and remark that each function from $\Hat W^{1,2}_0(V_{1,(0,a)})$
can be extended by zero to a function in 		$\Hat W^{1,2}_0(V_{1,(0,\alpha a)})$,
so the min-max principle implies
\[
\lambda_j(T^{1,D}_{\alpha^{1-q},(0,\alpha a)})\le \lambda_j(T^{1,D}_{\alpha^{1-q},(0,a)}).
\]
Using Corollary \ref{corol14} with $\eps:=\alpha^{1-q}$ and $\alpha\to+\infty$ we obtain
\begin{align*}
\lambda_j(T^{1,D}_{\alpha^{1-q},(0,a)})&=\left(\dfrac{A_\omega}{\alpha^{1-q}}\right)^{\frac{2}{2-q}}\lambda_j(L_1)+O\left(\dfrac{1}{\alpha^{1-q}}\right)^{\frac{q}{2-q}}
=A_\omega^\frac{2}{2-q}\alpha^\frac{2q-2}{2-q}\lambda_j(L_1)+O\left(\alpha^{\frac{q^2-q}{2-q}}\right),
\end{align*}
and the preceding inequalities yield
\begin{align*}
	\lambda_j(R^\Omega_\alpha)\le \alpha^2 \lambda_j(T^{1,D}_{\alpha^{1-q},(0,a)})&=\alpha^2\bigg[ A_\omega^\frac{2}{2-q}\alpha^\frac{2q-2}{2-q}\lambda_j(L_1)+O\Big(\alpha^{\frac{q^2-q}{2-q}}\Big)\bigg]\\
	&=A_\omega^\frac{2}{2-q} \lambda_j(L_1) \alpha^\frac{2}{2-q}+O\Big(\alpha^{\frac{2}{2-q}-(q-1)}\Big). \qedhere
\end{align*}
\end{proof}

Obtaining a lower bound requires slightly more work.	
\begin{lemma}\label{lem16}	For any $j\in\N$ one has
	\[
	\lambda_j(R^\Omega_\alpha)\ge A_\omega^\frac{2}{2-q} \lambda_j(L_1) \alpha^\frac{2}{2-q}+O\Big(\alpha^{\frac{2}{2-q}-(q-1)}\Big)\text{ for $\alpha\to +\infty$.}
	\]
\end{lemma}	

\begin{proof}
Let $j\in\N$ be fixed. For any $u\in W^{1,2}(\Omega)$ and any $\alpha>0$ we have
\[
r^\Omega_\alpha(u,u)\ge t^{\alpha,N}_{1,(0,a)}(u|_{\Omega_a},u|_{\Omega_a})+r^{\Omega'_a}_\alpha (u|_{\Omega'_a},u|_{\Omega'_a}),
\]
and the min-max principle shows that for any $\alpha>0$ it holds
\[
	\lambda_j(R^\Omega_\alpha)\ge \lambda_j\big(T^{\alpha,N}_{1,(0,a)}\oplus R^{\Omega'_a}_\alpha\big).
\]
Remark that
\[
\lambda_j\Big(T^{\alpha,N}_{1,(0,a)}\oplus R^{\Omega'_a}_\alpha\Big)\ge \min\Big\{\lambda_j\big(T^{\alpha,N}_{1,(0,a)}\big),\lambda_1(R^{\Omega'_a}_\alpha)\Big\}.
\]
As $\Omega'_a$ is a bounded Lipschitz domain, there is $c>0$ such that $\lambda_1(R^{\Omega'_a}_\alpha)\ge-c\alpha^2$ for all sufficiently large $\alpha$ (as discussed in the introduction). Due to the scaling \eqref{lscal} we have
\[
\lambda_j\big(T^{\alpha,N}_{1,(0,a)}\big)=\alpha^2 \lambda_j\big(T^{1,N}_{\alpha^{1-q},(0,a\alpha)}\big),
\]
and by putting all together we arrive at
\begin{equation}
	\label{lb3}
	\lambda_j(R^\Omega_\alpha)\ge \min\Big\{\lambda_j\big(T^{1,N}_{\alpha^{1-q},(0,a\alpha)}\big),-c\Big\}\alpha^2
	\text{ for all sufficiently large $\alpha$.}
\end{equation}

From now on we assume that $\alpha>1$. To study the eigenvalues of $T^{1,N}_{\alpha^{1-q},(0,a\alpha)}$
we pick smooth functions $\chi_1,\chi_2\in C^\infty (0,\infty)$ such that
	\[
	\chi_1(s)=0 \text{ for all }s>\frac{2a}{3},\qquad \chi_2(s)=0 \text{ for all }s<\frac{a}{3},
	\qquad
	\chi_1^2+\chi_2^2=1.
	\]	
Due to Lemma \ref{lem5} the set $W^{1,2}_{(0,\infty)}(V_{\alpha^{1-q},(0,a\alpha)})$ is a core domain of $t^{1,N}_{\alpha^{1-q},(0,a\alpha)}$, and for any $u\in W^{1,2}_{(0,\infty)}(V_{\alpha^{1-q},(0,a\alpha)})$ one has
\begin{gather*}
	(x_1,x')\mapsto \chi_1(x_1) u(x_1,x')\in W^{1,2}_{(0,a)}(V_{\alpha^{1-q},(0,a)}),\\
	(x_1,x')\mapsto \chi_2(x_1) u(x_1,x')\in W^{1,2}(V_{\alpha^{1-q},(\frac{a}{3},a\alpha)}),\\
	\|\chi_1 u\|^2_{L^2(V_{\alpha^{1-q},(0,a)})}+\|\chi_2 u\|^2_{L^2(V_{\alpha^{1-q},(\frac{a}{3},a\alpha)})}=\|u\|^2_{L^2(V_{\alpha^{1-q},(0,a\alpha)})},\\
\begin{aligned}
	t^{1,N}_{\alpha^{1-q},(0,a\alpha)}(u,u)&=	t^{1,D}_{\alpha^{1-q},(0,a)}(\chi_1 u,\chi_1 u)+
t^{1,N}_{\alpha^{1-q},(\frac{a}{3},a\alpha)}(\chi_2 u,\chi_2 u)\\ &\quad-\int_{V_{\alpha^{1-q},(0,a\alpha)}} \big(|\nabla\chi_1|^2+|\nabla\chi_2|^2\big)u^2\dd \H^d,
\end{aligned}
\end{gather*}
and by denoting $B:=\big\||\nabla\chi_1|^2+|\nabla\chi_2|^2\big\|_\infty$
and using the min-max-principle we arrive at
\begin{align*}
	\lambda_j\big(T^{1,N}_{\alpha^{1-q},(0,a\alpha)}\big)&\ge \lambda_j\Big(
	T^{1,D}_{\alpha^{1-q},(0,a)}\oplus T^{1,N}_{\alpha^{1-q},(\frac{a}{3},a\alpha)}\Big)-B\\
	&\ge \min\Big\{\lambda_j\Big(
	T^{1,D}_{\alpha^{1-q},(0,a)}\Big),\lambda_1\Big(
	T^{1,N}_{\alpha^{1-q},(\frac{a}{3},a\alpha)}\Big)
	\Big\}-B.
\end{align*}
For some fixed $C>0$ and all sufficiently large $\alpha$ we have, due to Lemma \ref{lem8} and Corollary \ref{corol14},
\[
\lambda_j\Big(
T^{1,D}_{\alpha^{1-q},(0,a)}\Big)=
A_\omega^\frac{2}{2-q}\alpha^\frac{2q-2}{2-q}\lambda_j(L_1)+O\Big(\alpha^{\frac{q^2-q}{2-q}}\Big)
< -C\alpha^{q-1}\le \lambda_1\Big(
T^{1,N}_{\alpha^{1-q},(\frac{a}{3},a\alpha)}\Big),
\]	
because
\[
\dfrac{2q-2}{2-q}=\dfrac{2}{2-q} (q-1)>q-1.
\]
This yields for all large $\alpha$ 
\[
\lambda_j\big(T^{1,N}_{\alpha^{1-q},(0,a\alpha)}\big)\ge \lambda_j\Big(
T^{1,D}_{\alpha^{1-q},(0,a)}\Big)-B=A_\omega^\frac{2}{2-q}\alpha^\frac{2q-2}{2-q}\lambda_j(L_1)+O\Big(\alpha^{\frac{q^2-q}{2-q}}\Big),
\]
and using \eqref{lb3} we obtain
\begin{align*}
		\lambda_j(R^\Omega_\alpha)&\ge \min\Big\{\lambda_j\big(T^{1,N}_{\alpha^{1-q},(0,a\alpha)}\big),-c\Big\}\alpha^2\ge  \min\Big\{A_\omega^\frac{2}{2-q}\alpha^\frac{2q-2}{2-q}\lambda_j(L_1)+O\Big(\alpha^{\frac{q^2-q}{2-q}}\Big),-c\Big\}\alpha^2\\
		&=\bigg[ A_\omega^\frac{2}{2-q}\alpha^\frac{2q-2}{2-q}\lambda_j(L_1)+O\Big(\alpha^{\frac{q^2-q}{2-q}}\Big) \bigg]\alpha^2=A_\omega^\frac{2}{2-q} \lambda_j(L_1) \alpha^\frac{2}{2-q}+O\Big(\alpha^{\frac{2}{2-q}-(q-1)}\Big). \qedhere
\end{align*}
\end{proof}

The claim of Theorem~\ref{thmain} follows by combining the upper bound
of Lemma~\ref{lem15} with the lower bound of Lemma~\ref{lem16}.

\section{Proof of Theorem \ref{thmeig} (localization of eigenfunctions)}\label{eigenfunctions}

Let $j\in\N$ and $b>0$ be fixed and $u$ be an eigenfunction corresponding to the eigenvalue
\[
E:=\lambda_j(R_\alpha^\Omega).
\]
For any function $\Phi\in C^\infty(\Omega)$ which is bounded and has bounded partial derivatives a direct computation gives
\begin{align*}
\int_\Omega|\nabla(e^{\Phi}u)|^2\ \dd \H^d&-\alpha\int_{\partial\Omega}|e^{\Phi}u|^2\ \dd \H^{d-1}
\\
&=\int_\Omega \big( |\nabla \Phi|^2 |e^{\Phi} u|^2+2e^{2\Phi}u\langle\nabla\Phi,\nabla u\rangle_{\R^d}+e^{2\Phi}|\nabla u|^2\big) \dd \H^d-\alpha\int_{\partial\Omega}|e^{\Phi}u|^2\ \dd \H^{d-1}
\\
&=\int_\Omega \big(|\nabla\Phi|^2+E\big)|e^\Phi u|^2\ \dd \H^d,
\end{align*}
where for the last equality we used that for $v:=e^{\Phi}u\in W^{1,2}(\Omega)$ it holds that
\[
r_\alpha^\Omega(u,v)=\int_\Omega \langle \nabla u,\nabla v\rangle_{\R^d}\, \dd \H^d-\alpha\int_{\partial\Omega}uv\, \dd \H^{d-1}=E\int_{\Omega}uv\, \dd \H^d
\]
due to the definition of $R^\Omega_\alpha$. We now set
\[
\Phi(x):=b\, \alpha|x|,
\]
then
$|\nabla\Phi|=b\, \alpha$. In virtue of Theorem \ref{thmain} one finds
$\alpha_0>1$ and $c_1>0$ such that $E<-c_1\alpha^{\frac{2}{2-q}}$ for all $\alpha>\alpha_0$. Plugging this into the above equation gives
\begin{equation}
	\label{agm1}
\begin{aligned}
\int_\Omega\big|\nabla(e^{\Phi}u)\big|^2\, \dd \H^d-\alpha\int_{\partial\Omega}|e^{\Phi}u|^2\, \dd \H^{d-1}&=\int_\Omega \big(|\nabla\Phi|^2+E\big)|e^\Phi u|^2\, \dd \H^d\\
&\le \int_\Omega (b^2\alpha^2-c_1\alpha^{\frac{2}{2-q}})|e^{\Phi}u|^2\, \dd \H^d
\end{aligned}
\end{equation}
for all $\alpha>\alpha_0$. Now we choose an IMS-type partition of unity $\{\chi_1,\chi_2\}$, i.e. two functions $\chi_1,\chi_2\in C^\infty(\R^d)$ such that
\begin{itemize}
	\item $\chi^2_1+\chi_2^2=1$,
	\item $\chi_1(x_1,x')=1$ for $\max\big\{|x_1|,|x'|\big\}\le 1$,
	\item $\chi_2(x_1,x')=1$ for $\max\big\{|x_1|,|x'|\big\}\ge 2$,
\end{itemize}
and denote
\[
\chi_{k,\alpha}(x):=\chi_k(\alpha x) \text{ for }k\in\{1,2\}.
\]
Then a direct computation gives
\begin{multline*}
	\int_\Omega\big|\nabla(e^{\Phi}u)\big|^2\, \dd \H^d-\alpha\int_{\partial\Omega}|e^{\Phi}u|^2\, \dd \H^{d-1}\\
	=\sum_{k=1}^2\Big(\int_{\Omega}|\nabla(\chi_{k,\alpha}e^{\Phi}u)|^2\, \dd \H^d-\alpha\int_{\partial\Omega}|\chi_{k,\alpha}e^{\Phi}u|^2\ \dd \H^{d-1}\Big)
	\\-\int_\Omega \big(|\nabla \chi_{1,\alpha}|^2+|\nabla \chi_{2,\alpha}|^2\big)|e^\Phi u|^2\,\dd \H^d,
\end{multline*}
which together with \eqref{agm1} yields
\begin{multline}
	\label{agm2}
\sum_{k=1}^2\Big(\int_{\Omega}|\nabla(\chi_{k,\alpha}e^{\Phi}u)|^2\, \dd \H^d
-\alpha\int_{\partial\Omega}|\chi_{k,\alpha}e^{\Phi}u|^2\, \dd \H^{d-1}\Big)\\
\le \int_{\Omega}(b_0\alpha^2+b^2\alpha^2-c_1\alpha^{\frac{2}{2-q}})|e^{\Phi}u|^2\, \dd \H^d,
\end{multline}
where $b_0:=\big\||\nabla\chi_1|^2+|\nabla\chi_2|^2\big\|_\infty$.

Let $a$ and $\Omega_a'$ be chosen as in Section \ref{secproof}. We can increase $\alpha_0$ and find $c_2,c_3,c_4>0$ such that for all $\alpha>\alpha_0$ one obtains:
\begin{itemize}
\item $\lambda_1(T_{1,(\frac{1}{\alpha},a)}^{\alpha,N})\ge -c_2\,\alpha^{q+1}$, see Lemma \ref{lem8} and \eqref{lscal},
\item $\lambda_1(R^{\Omega_a'}_\alpha)\ge -c_3\,\alpha^2$, see \eqref{lipschitz},
\item $\lambda_1(T_{1,(0,\frac{2}{\alpha})}^{\alpha,D})\ge -c_4\,\alpha^{\frac{2}{2-q}}$, see Lemma \ref{lem12} combined with \eqref{lscal}.
\end{itemize}
We now combine these inequalities with \eqref{agm2}:
\begin{align*}
\int_{\Omega}(b_0\alpha^2+&b^2\alpha^2-c_1\alpha^{\frac{2}{2-q}})|e^{\Phi}u|^2\, \dd \H^d\ge \sum_{k=1}^2\bigg(\int_{\Omega}|\nabla(\chi_{k,\alpha}e^{\Phi}u)|^2\, \dd \H^d-\alpha\int_{\partial\Omega}|\chi_{k,\alpha}e^{\Phi}u|^2\,\dd \H^{d-1}\bigg)
\\
\ge&\int_{V_{1,(0,\frac{2}{\alpha})}}|\nabla(\chi_{1,\alpha}e^{\Phi}u)|^2\, \dd \H^d-\alpha\int_{\partial_0 V_{1,(0,\frac{2}{\alpha})}}|\chi_{1,\alpha}e^{\Phi}u|^2\,\dd \H^{d-1}
\\
&+\int_{V_{1,(\frac{1}{\alpha},a)}}|\nabla(\chi_{2,\alpha}e^{\Phi}u)|^2\, \dd \H^d-\alpha\int_{\partial V_{1,(\frac{1}{\alpha},a)}}|\chi_{2,\alpha}e^{\Phi}u|^2\, \dd \H^{d-1}
\\
&+\int_{\Omega_a'}|\nabla(\chi_{2,\alpha}e^{\Phi}u)|^2\, \dd \H^d-\alpha\int_{\partial \Omega_a'}|\chi_{2,\alpha}e^{\Phi}u|^2\, \dd \H^{d-1}
\\
=&\, t^{\alpha,D}_{1,(0,\frac{2}{\alpha})}(\chi_{1,\alpha}e^{\Phi}u|_{V_{1,(0,\frac{2}{\alpha})}},\chi_{1,\alpha}e^{\Phi}u|_{V_{1,(0,\frac{2}{\alpha})}})+t^{\alpha,N}_{1,(\frac{1}{\alpha},a)}(\chi_{2,\alpha}e^{\Phi}u|_{V_{1,(\frac{1}{\alpha}),a}},\chi_{2,\alpha}e^{\Phi}u|_{V_{1,(\frac{1}{\alpha},a)}})
\\
&+r^{\Omega_a'}_\alpha(\chi_{2,\alpha}e^{\Phi}u|_{\Omega_a'},\chi_{2,\alpha}e^{\Phi}u|_{\Omega_a'})
\\
\ge& \lambda_1(T_{1,(0,\frac{2}{\alpha})}^{\alpha,D})\ \|\chi_{1,\alpha}e^{\Phi}u\|^2_{L^2(V_{1,(0,\frac{2}{\alpha})})}+\lambda_1(T_{1,(\frac{1}{\alpha},a)}^{\alpha,N})\ \|\chi_{2,\alpha}e^{\Phi}u\|^2_{L^2(V_{1,(\frac{1}{\alpha},a)})}
\\
&+\lambda_1(R^{\Omega_a'}_\alpha) \ \|\chi_{2,\alpha}e^{\Phi}u\|^2_{L^2(\Omega_a')}
\\
\ge& -c_4\, \alpha^{\frac{2}{2-q}} \, \|\chi_{1,\alpha}e^{\Phi}u\|^2_{L^2(\Omega)}-(c_2\, \alpha^{q+1}+c_3 \alpha^2) \, \|\chi_{2,\alpha}e^{\Phi}u\|^2_{L^2(\Omega)}.
\end{align*} 
By representing on the left-hand side $|e^{\Phi}u|^2=|\chi_{1,\alpha}e^{\Phi}u|^2+|\chi_{2,\alpha}e^{\Phi}u|^2$
we can rearrange the terms in the preceding inequality to obtain
\[
\big((b^2+b_0)\alpha^2+c_4\alpha^{\frac{2}{2-q}}\big)\int_{\Omega}|\chi_{1,\alpha}e^{\Phi}u|^2\, \dd \H^d
\ge \big(c_1\alpha^{\frac{2}{2-q}}-c_2 \alpha^{q+1}-(b^2+b_0+c_3)\alpha^2\big)\int_{\Omega}|\chi_{2,\alpha}e^{\Phi}u|^2\,\dd \H^d.
\]
Pick any $c_5\in (0,c_1)$ and $c_6>c_4$, then one can increase the value of $\alpha_0$
to obtain for all $\alpha>\alpha_0$
\[
c_6 \alpha^{\frac{2}{2-q}}\ge (b^2+b_0)\alpha^2+c_4\alpha^{\frac{2}{2-q}},
\quad
c_1\alpha^{\frac{2}{2-q}}-c_2 \alpha^{q+1}-(b^2+b_0+c_3)\alpha^2\ge c_5 \alpha^{\frac{2}{2-q}},
\]
hence,
\[
c_6 \alpha^{\frac{2}{2-q}}\int_{\Omega}|\chi_{1,\alpha}e^{\Phi}u|^2\, \dd \H^d
\ge c_5\  \alpha^{\frac{2}{2-q}}\int_{\Omega}|\chi_{2,\alpha}e^{\Phi}u|^2\, \dd \H^d,
\]
i.e. $\|\chi_{2,\alpha}e^{\Phi}u\|^2_{L^2(\Omega)}\le c_7\|\chi_{1,\alpha}e^{\Phi}u\|^2_{L^2(\Omega)}$  with $c_7:=c_6/c_5$,
therefore,
\begin{equation}
	\label{agm3}
\int_{\Omega}e^{2b\alpha|x|}\big|u(x)\big|^2\dd x=\|\chi_{1,\alpha} e^{\Phi}u\|^2_{L^2(\Omega)}+\|\chi_{2,\alpha} e^{\Phi}u\|^2_{L^2(\Omega)}
\le (1+c_7)\|\chi_{1,\alpha} e^{\Phi}u\|^2_{L^2(\Omega)}.
\end{equation}

Now remark that for any $x=(x_1,x')\in\mathrm{supp}\,\chi_{1,\alpha}$ we have $\max\big\{|x_1|,|x'|\big\}\le\frac{2}{\alpha}$
and $|x|\le \frac{2\sqrt{2}}{\alpha}$, which implies
\[
\|\chi_{1,\alpha} e^{\Phi}u\|^2_{L^2(\Omega)}
\le \int_{\Omega}e^{2b\sqrt{2}}|u|^2\, \dd \H^d,
\]
and the substitution into \eqref{agm3} gives the required estimate \eqref{agm0} with $B:=(1+c_7) e^{2b\sqrt{2}}$.

\appendix


\section{Spectral properties of the effective one-dimensional Schr\"odinger operator $L_\mu$}\label{appa}

For the sake of completeness we add proofs for the claims made in Section \ref{sec1d} about the operator $L_\mu$.
Recall that $L_\mu$ (with $\mu>0$) is the self-adjoint operator in $L^2(0,\infty)$ obtained as the Friedrichs extension
of the symmetric operator
\[
	C_c^\infty(0,\infty) \ni f \mapsto -f'' + V_\mu f, \qquad
	V_\mu(s):=\frac{q^2(d-1)^2-2q(d-1)}{4s^2}-\frac{\mu}{s^q}.
\]

\begin{lemma}\label{ess1} $[0,\infty)\subset\specess L_\mu$.
\end{lemma}
\begin{proof}
Choose cut-off functions $\chi_n$, with $n\in\N$, such that
\begin{itemize}
	\item $\chi_n\in C^\infty(\R)$ with $0\le \chi_n\le 1$,
	\item $\mathrm{supp}\,\chi_n\subset\big[n\pi,(2n+1)\pi\big]$,
	\item $\chi_n(s)=1$ for all $s\in\big[(n+1)\pi,2n\pi\big]$,
	\item $\|\chi_n'\|_\infty+\|\chi_n''\|_\infty\le c$ with some $c$ independent of $n$.
\end{itemize}
Let $k\in\R$ and define functions $f_n\in C^\infty_c(0,\infty)$ by
\[
f_n(s):=\chi_n(s)\cos(ks),
\]
then $\mathrm{supp}\, f_n\subset \big[n\pi,(2n+1)\pi\big]$ and 
\[
(L_\mu-k^2)f_n (s)= 2k\chi'_n(s)\sin(ks)-\chi''_n(s)\cos(ks)+V_\mu(s) f_n(s).
\]
Remark that the first two summands on the right-hand side are uniformly bounded in $(s,n)$ and supported in $\big[n\pi,(n+1)\pi\big]\cup\big[2n\pi,(2n+1)\pi\big]$. Furthermore, we can roughly estimate $\big|V_\mu(s)\big|\le \frac{c_0}{s}$ for all $s\ge \pi$ with some $c_0>0$ independent of $s$. These observations yield
\[
\|(L_\mu-k^2)f_n\|_{L^2(0,\infty)}\le c_1+\dfrac{c_0}{n}\|f_n\|_{L^2(0,\infty)}
\]
with some fixed $c_1>0$. On the other hand,
\[
\|f_n\|^2_{L^2(0,\infty)}\ge \int_{(n+1)\pi}^{2n\pi}\big|\cos(ks)\big|^2\ \dd s=(n-1)\int_0^\pi\big|\cos(ks)\big|^2\ \dd s
=(n-1)\dfrac{\pi}{2},
\]
hence,
\[
\dfrac{\|(L_\mu-k^2)f_n\|_{L^2(0,\infty)}}{\|f_n\|_{L^2(0,\infty)}}\xrightarrow{n\to\infty}0,
\]
which shows $k^2\in\spec\, L_\mu$. As $k$ was an arbitrary real number, one obtains $[0,\infty)\subset \spec L_\mu$.
As each non-isolated point of the spectrum belongs to the essential spectrum, one arrives at the conclusion.
\end{proof}

\begin{lemma} $\inf\specess L_\mu\ge 0$.
\end{lemma}

\begin{proof}
Let $\chi_1,\chi_2\in C^\infty(0,\infty)$ with $\chi_1^2+\chi_2^2=1$, $\chi_1(s)=1$ for $s<1$ and $\chi_1(s)=0$ for $s>2$.
For $t>0$ denote
\[
\chi_{k,t}:=\chi_k\Big(\dfrac{\cdot}{t}\Big),\quad k\in\{1,2\},
\]
then for any $f\in C^\infty_c(0,\infty)$ one has, with $a:=\big\|(\chi_1')^2+(\chi_2')^2\big\|_\infty$,
\begin{align*}
	\langle f,L_\mu f\rangle_{L^2(0,\infty)}&=\int_0^\infty \big( |f'(s)|^2+V_\mu |f(s)|^2\big)\,\dd s\\
	&=\int_0^\infty \big( \big|(\chi_{1,t}f)'(s)\big|^2+V_\mu \big|(\chi_{1,t}f)(s)\big|^2\big)\,\dd s\\
	&\quad +\int_0^\infty \Big( |(\chi_{2,t}f)'(s)|^2+V_\mu |(\chi_{2,t}f)(s)|^2\big)\,\dd s\\
	&\quad -\int_0^\infty \Big(\big|\chi_{1,t}'(s)\big|^2+\big|\chi_{2,t}'(s)\big|^2\Big)\big|f(s)\big|^2\dd s	\\
	&\ge \int_0^{2t} \big( \big|(\chi_{1,t}f)'(s)\big|^2+V_\mu \big|(\chi_{1,t}f)(s)\big|^2\big)\,\dd s\\
	&\quad +\int_t^\infty \Big( |(\chi_{2,t}f)'(s)|^2+V_\mu |(\chi_{2,t}f)(s)|^2\big)\,\dd s
	-\dfrac{a}{t^2}\|f\|^2_{L^2(0,\infty)}\\
	&=\Big\langle \chi_{1,t}f, L_{\mu,2t}(\chi_{1,t}f)\Big\rangle_{L^2(0,2t)}+
	\Big\langle \chi_{2,t}f, \Tilde L_{\mu,t}(\chi_{2,t}f)\Big\rangle_{L^2(t,\infty)}
\end{align*}
where $L_{\mu,2t}$ is defined in Section~\ref{sec1d} and $\Tilde L_{\mu,t}$ is the self-adjoint operator in $L^2(t,\infty)$
defined as the Friedrichs extension of
\[
f\mapsto -f''+V_\mu f,\quad f\in C^\infty_c(t,\infty).
\]
Then the min-max principle implies that for any $t>0$ one has the inequality
\[
\inf \specess L_\mu\ge \inf\specess (L_{\mu,2t}\oplus\Tilde L_{\mu,t})-\frac{a}{t^2}.
\]
Using the fact that $L_{\mu,2t}$ has compact resolvent, so $\specess L_{\mu,2t}=\emptyset$, this reduces to
\begin{align*}
\inf \specess L_\mu&\ge \inf\specess \Tilde L_{\mu,t}-\frac{a}{t^2}
\ge \inf\spec \,\Tilde L_{\mu,t}-\frac{a}{t^2}\ge \inf_{s> t} V_\mu(s)-\frac{a}{t^2} \text{ for all } t>0.
\end{align*}
For $t\to\infty$ the right-hand side converges to $0$, which gives the sought lower bound.
\end{proof}

By combining the two preceding assertions we obtain:
\begin{corollary} $\specess L_\mu=[0,\infty)$.
\end{corollary}

Having determined the essential spectrum, we now turn to the discrete spectrum of $L_\mu$. 
\begin{lemma}
The operator $L_\mu$ has infinitely many negative eigenvalues.
\end{lemma}
\begin{proof}
	In view of the specific form of $V_\mu$ one can choose $R>0$ and $c>0$ such that for all $s> R$ it holds $V_\mu(s)\le -c/s^q$. Now pick any $\varphi\in C^\infty_c(0,\infty)$ with $\mathrm{supp}\,\varphi\subset (R,2R)$ and $\|\varphi\|_{L^2(0,\infty)}=1$, and for any $t>1$ consider the functions
	\[
	\varphi_t:\ s\mapsto\dfrac{1}{\sqrt{t}}\varphi\left(\frac{s}{t}\right),
	\]
	then $\varphi_t\in C^\infty_c(0,\infty)$ with $\|\varphi_t\|_{L^2(0,\infty)}=1$ and $\mathrm{supp}\,\varphi_t\subset  (tR,2tR)$. One estimates
\begin{align*}
\int_0^\infty \big| \varphi_t'(s)\big|^2\, \dd s &= \int_0^\infty \dfrac{1}{t^{3}} \Big| \varphi' \Big(\dfrac{s}{t}\Big) \Big|^2 \, \dd s=\dfrac{1}{t^2}\underbrace{\int_0^\infty \big|\varphi'(s)\big|^2\, \dd s}_{=:a>0} \equiv \dfrac{a}{t^2},\\
\int_0^\infty V_\mu(s)\big|\varphi_t(s)\big|^2 \dd s&=\dfrac{1}{t}\int_{t R}^{2 t R} V_\mu(s)\Big|\varphi\Big(\dfrac{s}{t}\Big)\Big|^2\, \dd s\le \dfrac{1}{t} \int\limits_{t R}^{2 t R} \bigg(-\dfrac{c}{|s|^q}\bigg)\Big|\varphi\Big(\dfrac{s}{t}\Big)\Big|^2\, \dd s\\
&=-\dfrac{c}{t^q} \underbrace{\int_R^{2 R} \dfrac{\big|\varphi(y)\big|^2}{|y|^q}\, \dd y}_{=:b>0}=-\dfrac{bc}{t^q}.
\end{align*}
As $q<2$, one can choose $n_0\in\N$ sufficiently large to have
\begin{align*}
 \langle\varphi_t, L_\mu\varphi_t\rangle_{L^2(0,\infty)}&=\int_0^\infty  \Big(\big|\varphi_t'(s)\big|^2+V_\mu(s)\big|\varphi_t(s)\big|^2 \Big)\, \dd s\\
&\le\dfrac{a}{t^2} -\dfrac{bc}{t^q}=\dfrac{a-bc t^{2-q}}{t^2}<0 \text{ for all } t\ge n_0.
\end{align*}
Now for $n\ge n_0$ put $\psi_n:=\varphi_{2^n }$, then the functions $\psi_n$ have mutually
disjoint supports and, therefore, form an orthonormal family, and
\[
\langle \psi_m, L_\mu\psi_n\rangle_{L^2(0,\infty)} =0 \text{ for $m\ne n$}, \quad a_n:=\langle \psi_n, L_\mu\psi_n\rangle <0.
\]

Let $N\in \N$ and consider the $N$-dimensional subspace $F:=\mathrm{span}\{\psi_{n_0+1},\dots,\psi_{n_0+N}\}$. 
Let $\psi\in F$, then
\begin{align*}
\psi&=\sum_{k=1}^{N} \xi_k \psi_{n_0+k}, \quad \xi=(\xi_1,\dots,\xi_N)\in \R^N,\qquad
\|\psi\|^2_{L^2(0,\infty)}=\sum_{k=1}^{N} |\xi_n|^2,\\
 \langle\psi,L_\mu\psi\rangle_{L^2(0,\infty)}&=\sum_{k,k'=1}^{N} \langle\psi_{n_0+k}, L_\mu\psi_{n_0+k'}\rangle_{L^2(0,\infty)}\xi_k\xi_{k'}=\sum_{k=1}^{N} a_{n_0+k}|\xi_k|^2 
\\ 
&\le \underbrace{\max\{a_{n_0+1},\dots,a_{n_0+N}\}}_{=:b_N<0}
\sum_{k=1}^N |\xi^2_k|= b_N \|\psi\|^2_{L^2(0,\infty)}.
\end{align*}
It follows that
\[
\sup_{\psi\in F,\,\psi\ne 0}\dfrac{\langle \psi, L_\mu\psi\rangle_{L^2(0,\infty)}}{ \|\psi\|^2_{L^2(0,\infty)}}\le b_N<0=\inf\specess L_\mu,
\]
and the min-max principle implies that $L_\mu$ has at least $N$ eigenvalues in $(-\infty,0)$.
As $N\in\N$ was arbitrary, the result follows.
\end{proof}

\begin{lemma}
	All eigenvalues of $L_\mu$ are simple.
\end{lemma}

\begin{proof}
	This is a rather standard  property of second order ordinary differential operators \cite{weid}, but we give 
	an explicit argument adapted to this particular case.	Let $f$ be an eigenfunction of $L_\mu$ for an eigenvalue $E$, then $f$ is a solution of the differential equation
	\begin{equation}
		-f''+V_\mu f=E f \label{vme}
	\end{equation}
	which is square-integrable on $(0,\infty)$, in particular, on $(1,\infty)$. Remark that $V_\mu$ is bounded at infinity, which implies that we have the so-called limit point case at infinity, see e.g. \cite[Theorem 6.6]{weid}. It follows, in particular, that for any $E\in\R$ the space of solutions to \eqref{vme} 
	that are square-integrable on $(1,\infty)$ is at most one-dimensional, see e.g.~\cite[Theorem 5.6]{weid}.
\end{proof}

\section*{Acknowledgments}
The authors are indebted to the referee for the careful reading and constructive remarks that motivated them to deepen the initial study and include Section \ref{eigenfunctions} and Appendix \ref{appa}.

F.~Sk is supported by the Alexander von Humboldt Foundation, Germany. M. Vogel is partially supported by the Deutsche
Forschungsgemeinschaft (DFG, German Research Foundation) – 471212562.

\end{document}